\pdfoutput=1
\RequirePackage{ifpdf}
\ifpdf % We~are running pdfTeX in pdf mode
\documentclass[pdftex]{sigma}%,draft
\else
\documentclass{sigma}
\fi

\usepackage{mathtools}

\input xy
\xyoption{all}

\numberwithin{equation}{section}

\newtheorem{Theorem}{Theorem}[section]
\newtheorem*{Theorem*}{Theorem}
\newtheorem{Corollary}[Theorem]{Corollary}
\newtheorem{Lemma}[Theorem]{Lemma}
\newtheorem{Proposition}[Theorem]{Proposition}
\newtheorem{Construction}[Theorem]{Construction}

 { \theoremstyle{definition}
\newtheorem{Definition}[Theorem]{Definition}

\newtheorem{Example}[Theorem]{Example}
\newtheorem{Remark}[Theorem]{Remark} }

\def\CC{\mathbb{C}}
\def\DD{\mathbb{D}}
\def\LL{\mathbb{L}}
\def\PP{\mathbb{P}}
\def\RR{\mathbb{R}}
\def\ZZ{\mathbb{Z}}

\def\cW{\mathcal{W}}
\def\calP{\mathcal{P}}
\def\calL{\mathcal{L}}
\def\cC{\mathcal{C}}
\def\cO{\mathcal{O}}
\def\cF{\mathcal{F}}
\def\cE{\mathcal{E}}

\newcommand\frX{\mathfrak{X}}
\newcommand\frs{\mathfrak{s}}

\newcommand{\Coh}{\textup{Coh}}
\newcommand{\Ind}{\textup{Ind}}
\newcommand\Loc{\textup{Loc}}
\newcommand{\QCoh}{\textup{QCoh}}
\newcommand\Spec{\textup{Spec}}
\newcommand\Hom{\textup{Hom}}

\newcommand\SL{\textup{SL}}

\newcommand\nc{\newcommand}
\nc\oX{\overline{X}}
\nc\oD{\overline{D}}
\nc\oT{\overline{T}}
\nc\bsig{\overline{\sigma}}
\nc\Tors{\on{Tors}}

\nc\on{\operatorname}
\nc\ol{\overline}
\nc\ul{\underline}
\nc\Bl{\on{Bl}}
\nc\Conv{\on{Conv}}
\nc\Int{\on{Int}}
\nc\Pc{\mathring{P}}

\nc\Cone{\on{Cone}}
\nc\colim{\varinjlim}
\nc\Sh{\mathrm{Sh}}
\nc\wsh{\Sh^w}
\nc\wmsh{\mu\Sh^w}
\nc\mmod{\on{-mod}}
\nc\ovi{\overline{i}}

\begin{document}
\allowdisplaybreaks

\newcommand{\arXivNumber}{2103.12232}

\renewcommand{\PaperNumber}{055}

\FirstPageHeading

\ShortArticleName{Mirror Symmetry for Truncated Cluster Varieties}

\ArticleName{Mirror Symmetry for Truncated Cluster Varieties}

\Author{Benjamin GAMMAGE~$^{\rm a}$ and Ian LE~$^{\rm b}$}

\AuthorNameForHeading{B.~Gammage and I.~Le}

\Address{$^{\rm a)}$~Department of Mathematics, Harvard University, USA}
\EmailD{\href{mailto:gammage@math.harvard.edu}{gammage@math.harvard.edu}}

\Address{$^{\rm b)}$~Mathematical Sciences Institute, Australian National University, Australia}
\EmailD{\href{mailto:ian.le@anu.edu.au}{ian.le@anu.edu.au}}

\ArticleDates{Received August 25, 2021, in final form July 15, 2022; Published online July 19, 2022}

\Abstract{In the algebraic setting, cluster varieties were reformulated by Gross--Hacking--Keel as log Calabi--Yau varieties admitting a toric model. Building on work of Shende--Treumann--Williams--Zaslow in dimension 2, we describe the mirror to the GHK construction in arbitrary dimension: given a truncated cluster variety, we construct a symplectic manifold and prove homological mirror symmetry for the resulting pair. We relate our construction to various aspects of cluster theory which are known to symplectic geometers.}


\Classification{53D37; 13F60}

\vspace{1mm}

\textit{Note on revised version: an earlier version of this paper contained an error in Section~{\rm 5}, which invalidated the proofs of some results in Sections~{\rm 5.2--5.3}; we have now removed that section of the paper, and we now discuss our expectations on the relation to toric geometry at the end of Section~{\rm \ref{subsec:future}}.}

\section{Introduction}
Cluster varieties are algebraic varieties which are glued together from torus charts using birational maps that are called \emph{cluster transformations}. They are described by starting with a~seed,~$\frs$, which we will define below. Each seed has an associated seed torus, which is one of the torus charts for the cluster variety. It also contains some combinatorial information which explains how to \emph{mutate} seeds. By mutating, one contains a (often infinite) set of seed tori. Mutation also gives a prescription for gluing these seed tori together. The result of gluing these tori together is called a \emph{cluster variety}. Often, one prefers to consider the affinization of this space, in which case it is this affinization which is referred to as the cluster variety.

Cluster varieties come in two flavors, $\mathcal{X}$ and $\mathcal{A}$, which have the same seed data, but use somewhat different birational transformations to glue seed tori. For each seed $\frs$, there is a~natu\-ral dual seed~$\frs^\vee$. Fock and Goncharov conjectured that the cluster varieties $\mathcal{X}(\frs)$ and~$\mathcal{A}\big(\frs^\vee\big)$ satisfied an intricate set of conjectures that they called the \emph{duality conjectures}. Moreover, they suspected that these duality conjectures were a consequence of homological mirror symmetry between $\mathcal{X}(\frs)$ and $\mathcal{A}\big(\frs^\vee\big)$.

In \cite{GHK2, GHK1}, Gross, Hacking, and Keel successfully placed the duality conjectures within the framework of mirror symmetry for log Calabi--Yau varieties. To start with, they reinterpreted the cluster mutation formula as the birational map given by an elementary transformation of $\PP^1$-bundles. This allowed them to give an algebro-geometric construction of cluster varieties: begin with a toric variety $\oX$ whose boundary is purely of codimension 1 (hence a disjoint union of tori); blow up a disjoint union of subtori (determined by the seed data) which are codimension~1 in the boundary tori; delete the strict transform of the original toric boundary.

Above dimension 2, the na\"ive result of performing this blow-up-and-delete operation described in \cite{GHK2} may produce a variety which is ``almost'' a cluster variety, in the sense that the former may differ in codimension 2 from the traditional Fock--Goncharov definition of cluster variety. (In particular, they have the same affinization.) In order to distinguish them, we refer to the former as a {\it truncated cluster variety}. These are the varieties whose mirrors we produce in this paper.

We describe here a mirror to the GHK construction: beginning with the Liouville sector (symplectic Landau--Ginzburg model) mirror to $\oX$, perform Weinstein handle attachments as indicated by the seed data. Using recent advances \cite{GPS2,GPS3,GPS1,NS20} in the theory of Weinstein Fukaya categories, we are able to compute the Fukaya categories of the resulting spaces and match them with categories of coherent sheaves on the truncated cluster varieties. We can thus express many features of cluster algebra quite naturally in the language of symplectic geometry. Many of these features have already been observed in the 2-dimensional case in \cite{STW,STWZ}, subsequently made explicit in \cite{HK20}; we offer a new perspective on the mirror-symmetry result proved there, as well as a generalization to arbitrary dimensions.

The basic Weinstein handle attachment we perform is a familiar part of two-dimensional symplectic geometry, in the form
of Symington's ``nodal trade'' \cite{Sym}. The local model for this move, its relation to cluster theory, and its higher-dimensional analogues have been previously studied \cite{Nad-mut,PT,Vianna}. However, our work is the first to describe the mirror to a general truncated cluster variety in terms of a series of these nodal trade operations.

\subsection{Mirror construction and plan of the paper}
Our construction is most indebted to \cite{STW}, which studied cluster structures inside the Fukaya category of a Weinstein 4-manifold $W$ obtained from a cotangent bundle $T^*\Sigma$ by Weinstein handle attachments. In the case when $\Sigma = T$ is a 2-torus, we propose to reintepret their construction in the framework of homological mirror symmetry, treating $T$ as the fiber of an~SYZ fibration: the wrapped Fukaya category $\cW(W)$ can be understood as the category of coherent sheaves $\Coh\big(W^\vee\big)$ on a mirror cluster variety.

The best framework for understanding this mirror cluster variety turns out to be the Gross--Hacking--Keel construction, which we review in Section~\ref{sec2}. Their construction, beginning with a~toric variety, performs a certain blow-up-and-delete operation in the toric divisors.
Our obser\-va\-tion is that if~$W^\vee$ is mirror to a Weinstein manifold obtained from the handle attachments described in \cite{STW}, the blow-up-and-delete operations of~$W^\vee$ are in bijection with the handle attachments of~$W$.

We prove that in arbitrary dimensions, the basic blow-up-and-delete operation of \cite{GHK2} is mirror to a basic Weinstein disk attachment.
A seed $\frs$, which in \cite{GHK2} is treated as a recipe for producing a cluster variety $U$, we treat as a recipe for producing a Weinstein manifold $U^\vee$.
After some review of microlocal sheaf theory in Section~\ref{sec3}, in Section~\ref{sec4} we describe this Weinstein manifold and prove our main theorem, a homological mirror symmetry result for this pair:
\begin{theorem*}
 There is an equivalence of dg categories $\Coh(U)\cong \cW\big(U^\vee\big)$ between the category of coherent sheaves on the truncated cluster variety $U$ and the wrapped Fukaya category of the Weinstein manifold $U^\vee$.
\end{theorem*}
We also mention that cluster mutations can be understood in our framework, referring to \cite{Nad-mut,PT,STW} for details.

\subsection[{Relation to [18]}]{Relation to \cite{HK20}}

In \cite{HK20}, Paul Hacking and Ailsa Keating produced an equivalence similar to the one we describe here, in the special case where the spaces involved are 2-dimensional. As we do, they follow~\cite{GHK2} in beginning with a log Calabi--Yau pair $(Y,D)$ obtained from certain non-toric blowups of a~toric pair, which they require to be 2-dimensional. (Our $U$ is the complement $Y \setminus D$.) They associate to this pair both a space $M$ which is mirror to $Y\setminus D$ together with a Lefschetz fibration $f\colon M\to \CC$ which is mirror to the compactification $Y\supset Y\setminus D$, and they prove a homological mirror symmetry equivalence in each case, which we may summarize as follows:
\begin{theorem*}[\cite{HK20}]
There is an equivalence of dg categories $\Coh(Y\setminus D)\cong \cW(M)$ between the category of coherent sheaves on the cluster variety $Y\setminus D$ and the wrapped Fukaya category of~$M$. In addition, there is an equivalence $\Coh(Y)\cong \cW(M,f)$ between coherent sheaves on the compactification and the Fukaya--Seidel category of the mirror Lefschetz fibration.
\end{theorem*}

In this 2-dimensional setting, it would be useful to compare the mirror $M$ described in \cite{HK20} to the space $U^\vee$ described in this paper. We believe it should be straightforward to check that the Lefschetz fibration on $M$ describes the same sequence of Weinstein handle attachments by which we have described $U^\vee$. Their work then also indicates the changes which would need to be made to our constructions in order to construct a mirror to the compact variety~$Y$ (rather than just the open part $Y\setminus D$). A detailed comparison of these constructions and the relevant homological mirror symmetry computations are interesting questions for future work.\looseness=1

Conversely, our work also suggests a generalization of the results of \cite{HK20}, since in this paper we allow ourselves to work in arbitrary dimensions rather than just dimension 2. Indeed, the main idea of this paper is that so long as we restrict ourselves on the B-side to the setting of truncated cluster varieties, there are no new phenomena besides those which have already appeared in dimension 2.
It would be very interesting to know whether this insight can be applied to produce mirror Lefschetz fibrations in higher dimensions as well.

\subsection{Future questions}\label{subsec:future}
As mentioned above, we restrict ourselves in this paper to varieties obtained from the construction of \cite{GHK2}, which may
differ in codimension 2 from the cluster varieties which have traditionally arisen in representation theory and geometry.
The mirrors to these latter cluster varieties will be Weinstein manifolds with skeleta more complicated than those described here.
In future work, we will describe these more complicated Lagrangian skeleta in the special cases of positroid cells and Richardson varieties, and explain their relation to mirror symmetry.

Those more complicated skeleta are also desirable because they admit deformations to more singular skeleta which are actually {\em holomorphic Lagrangian} in a certain hyperk\"ahler structure. (In the simplest example of a skeleton which is a torus with disk glued in, this deformation collapses the disk to produce a nodal curve.) Such skeleta have close connections with representation theory and holomorphic symplectic geometry; in the toric hyperk\"ahler setting, these skeleta have been used to prove mirror-symmetry results in \cite{GMW19,MW18}.

However, already in the cases we consider, it seems likely that the Lagrangian skeleta we describe can be used to produce interesting results in symplectic geometry. For instance, the skeleton of one of our Weinstein manifolds $U^\vee$ encodes the Reeb dynamics on the contact boundary of $U^\vee$. In~\cite{Pasc-SH}, an analysis of Reeb dynamics for $U^\vee$ 2-dimensional was used in order to study its quantum cohomology ring and hence the coordinate ring of the mirror cluster variety; we hope that our symplectic construction will help uncover new features of the algebraic geometry of mirror cluster varieties.

Finally, it would be interesting to identify the Weinstein manifolds described in this paper with K\"ahler manifolds naturally arising in toric geometry. Given a toric surface $\overline{X}_\Delta$ with moment polytope $\Delta,$ the \textit{nodal trade} construction of~\cite{Sym} can be used to construct an integrable system with singularities
\[
    \overline{X}_\Delta\to \Delta,
\]
where now the preimage of $\partial \Delta$ is a certain smoothing of the toric boundary. In the case of a~smoothing of a single node in $\partial \overline{X}_\Delta$, the complement $\overline{X}_\Delta\setminus \partial \overline{X}_\Delta$ can be shown to have skeleton given by a torus with a single disk attached. This case, and its relation to cluster theory, is studied in~\cite{PT}.
We expect that the Weinstein manifolds we describe, together with singular integrable systems on them, can be obtained by globalizing this nodal trade construction.

\subsection*{Notation}
The following notation will be used throughout the paper.
\begin{itemize}\itemsep=0pt
 \item We pick dual lattices $N$, $M$, and we write $N_\RR$, $N_\CC$ for $N\otimes_\ZZ\RR$, $N\otimes_\ZZ\CC$, and similarly for~$M$.
 \item We write $N_{S^1}$, $N_{\CC^\times}$ for the real and complex tori $N_\RR/N$, $N_\CC/N$, and similarly for $M$. Note that $N_{\CC^\times} = N_\RR/N\times N_\RR$. In sections of this paper where we are only concerned with one of $M$ or $N$ and no ambiguity can result, we will sometimes denote this real or complex torus by $T$ or $T_\CC$, respectively.
 \item We write $\Sigma$ for a stacky fan of rays in $N_\RR$. ``Stacky'' means that we keep track not only of the rays $\langle \psi_1\rangle, \dots, \langle \psi_r\rangle$, but of not necessarily primitive generators $\{d_i\psi_i\}_{i=1}^r$, where $\psi_i$ is primitive and $d_i\in \ZZ_{>0}$.
 \item We denote the toric stack with fan $\Sigma$ by $\oX_\Sigma$. It has toric divisor $\oD$ and dense open torus $\oT_{\CC} := N_{\CC^\times}$, which has cocharacter lattice $N$ and character lattice~$M$.
\end{itemize}

\subsection*{Categorical conventions}
Throughout this paper, we work with stable $\infty$-categories over a field $k$ of characteristic 0, which we model as dg categories (so that, for instance, $\Coh(X)$ always refers to the dg category of coherent sheaves on $X$). All functors are derived, and the word ``limits/colimit'' in a diagram of categories is always used to refer to a homotopy limit/colimit.

\section{Cluster varieties from non-toric blowups}\label{sec2}
In this section, we review the Gross--Hacking--Keel construction of cluster varieties, following~\cite{HK-icm} and then, with an eye toward our future symplectic constructions, we discuss their SYZ geometry.
\subsection{Review of the GHK construction}
The constructions of the Gross--Hacking--Keel program take place within the framework of log Calabi--Yau geometry.

\begin{Definition}
Let $(X,D)$ be a smooth projective variety equipped with a reduced normal crossings divisor $D$ satisfying $K_X+D=0$. In this case we say that $(X,D)$ is a {\em log Calabi--Yau pair} and $U=X\setminus D$ is a {\em log Calabi--Yau variety}. We write $\Omega$ for the unique (up to scalar multiples) holomorphic volume form on $U$ with simple poles on each component of $D$.
\end{Definition}

\begin{Remark}\label{rem:quasiproj-pair}
 In this paper, we will actually be interested only in the pair $(X',D')$ obtained from $(X,D)$ as above by deleting all strata in $D$ of codimension at least 2 in $X$. The resulting pair does not strictly fall under the above definition, since $X'$ is now only quasiprojective, but we will abuse notation slightly by continuing to refer to $(X,D)$ as a log Calabi--Yau pair.
\end{Remark}

\begin{Example}
 Let $\big(\oX,\oD\big)$ be an $n$-dimensional toric variety with its toric boundary divisor. Then $\big(\oX,\oD\big)$ is a log Calabi--Yau pair with $U=X\setminus D\cong (\CC^\times)^n$, $\Omega=\frac{{\rm d}x_1}{x_1}\wedge\cdots\wedge \frac{{\rm d}x_n}{x_n}$.
\end{Example}

As suggested in Remark~\ref{rem:quasiproj-pair}, we will only be interested here in considering log Calabi--Yau pairs where $D$ has no higher-codimension strata. Such a pair in the toric case can be obtained from a fan with no higher-dimensional cones.

\begin{Example}\label{ex:quasiproj-pair}
 Let $\Sigma\subset N_\RR$ be a fan of cones of dimension at most 1. In other words, $\Sigma$ contains the zero cone and the rays generated by primitive vectors $\psi_1,\dots,\psi_r\in N$. Then $\big(\oX_\Sigma,\oD\big)$ is a~log Calabi--Yau pair where $\oD = \bigsqcup_{i=1}^r \oD_i$ is a disjoint union of copies of $(\CC^\times)^{n-1}$. More precisely, the component $\oD_i$ corresponding to ray $\langle \psi_i\rangle$ is isomorphic to the quotient torus $(N/\psi_i)_{\CC^\times}$. Note that the character lattice of the torus $\oD_i$ is therefore the orthogonal $\psi_i^\perp\subset M$.
\end{Example}

\begin{Remark}\label{rem:stackiness}
 In fact, we will see that in order to produce cluster varieties whose seed matrix is not skew-symmetric but only skew-symmetrizable, we will have to allow $\Sigma$ to be a toric stack. For information on toric stacks, we refer to \cite{BCS}. In the setting of Example~\ref{ex:quasiproj-pair}, this means only that we equip $\Sigma$ with a choice of not necessarily primitive generator of each ray, or equivalently that to each primitive generator $\psi_i$, we associate a positive integer~$d_i$. The resulting toric stack~$\oX_\Sigma$ has coarse moduli space the usual toric variety associated to~$\Sigma$, but $\oX_\Sigma$ has a $\ZZ/d_i$ isotropy group along the divisor~$\oD_i$.
\end{Remark}

The discovery of \cite{GHK2,GHK1} is that log Calabi--Yau varieties related to the previous example in a simple way possess cluster structure.
\begin{Construction}\label{constr:ghk-cluster}
 Let $\big(\oX_\Sigma,\oD\big)$ be an $n$-dimensional toric pair as in Example~$\ref{ex:quasiproj-pair}$, and choose a~character $\chi_i\in \psi_i^\perp$ for each component $\oD_i$ of $\oD$.
 Let $H\subset \oD$ be the subvariety whose intersection with the component $\oD_i$ is equal to the linear subspace $\{\chi_i = -1\}$.
 We write $(X,D)$ for the blowup of $\oX$ along $H$ together with the strict transform of $\oD$.
\end{Construction}
\begin{Remark}
 The construction of \cite{GHK2} more generally allows components of $H$ to be defined by an equation of the form $\chi_i = \lambda$ for any $\lambda\in \CC^\times$. We prefer to restrict to the case of $\lambda = -1$, since this is a paper about exact symplectic geometry, and a mirror to the variety $X$ obtained for more general $\lambda$ is expected to be non-exact. (See \cite[Section~2.2]{HK20} for a more detailed discussion of this phenomenon in dimension 2.) However, note that in this paper we do not discuss mirror symmetry for the variety $X$ but only for $U$, which is not generally sensitive to the value of $\lambda$, so this issue will not arise explicitly in this paper.
\end{Remark}

Observe that if the divisor $\oD$ has a single component, then we can choose coordinates so that $\oX = (\CC\times \CC^\times)\times(\CC^\times)^{n-2}$ and $H=\{(0,-1)\}\times (\CC^\times)^{n-2}$, so that the space $X$ resulting from the GHK construction is
\[X = \Bl_{\{(0,1)\}}(\CC\times\CC^\times)\times (\CC^\times)^{n-2}.\]
A recurring theme in this paper is that Construction \ref{constr:ghk-cluster} can be reduced to repeated applications of this basic move, making this the most important (and also simplest) case to understand. We~begin to study it in the following example.

\begin{Example}\label{ex:bside-ex1}
 Let $\oX = \CC\times \CC^\times$ and $H=\{(0,-1)\}$, as above. Then we can describe the blowup $X = \Bl_H\oX$ as embedded in $\CC\times\CC^\times\times \PP^1$ as
 \[
 X \cong \{(x,y),[z:w]\mid xz-(y+1)w =0\}\subset \CC\times\CC^\times\times\PP^1.
 \]
 The strict transform $D$ of the toric boundary divisor $\oD = \{0\}\times \CC^\times$ contains the point $[1:0]$ in the exceptional $\PP^1$, and hence the variety $U=X\setminus D$ is completely contained in the affine chart $\CC\times\CC^\times\times \CC_z$:
 \[
 U \cong \{(x,y,z)\mid xz = y+1\}\subset \CC\times \CC^\times\times \CC.
 \]
 If desired, we can also eliminate the variable $y$ to describe $U$ as the complement
 \[U=\CC^2\setminus \{xz=1\}.\]
\end{Example}

\begin{Definition}\label{def:toricmodel}
 Let $U = X\setminus D$ for a log Calabi--Yau pair $(X,D)$ obtained via Construction~\ref{constr:ghk-cluster}. In this case we call $U$ a {\em truncated cluster variety}, and we say that the blowup $(X,D)\to \big(\oX,\oD\big)$ is a {\em toric model} for $U$.
\end{Definition}

\begin{Remark} Note that in case the dimension of $X$ is equal to 2, the data of the choice of character $\chi_i\in \psi_i^\perp$ in Construction~\ref{constr:ghk-cluster} is not necessary; knowledge of the divisors in which the blow-up is to take place is sufficient data in this case.
\end{Remark}

Composition of the inclusion $U\to X$ and the blow-up map $X\to \oX$ defines a birational morphism $U\to \oX$. The inverse birational map is well-defined on the dense torus $\oT\subset \oX$, giving an open inclusion $\oT\hookrightarrow U$, defining the toric chart on $U$ corresponding to toric model $(X,D)\to\big(\oX,\oD\big)$. We will see that, as expected, other toric charts can be reached by mutation; first, we recall the traditional defining data for a cluster variety.

\begin{Definition}\label{def:seed}
 A {\em seed} $\frs$ is the data of
 \begin{itemize}\itemsep=0pt
 \item the lattice $N$;
 \item a skew bilinear form $\bsig\in \wedge^2 M$;
 \item a basis $\{\psi_1,\dots,\psi_r,\psi_{r+1},\dots,\psi_n\}$ of $N$, along with a designated subset $\{\psi_1,\dots,\psi_r\}$ of {\em unfrozen} basis vectors;
 \item positive integers $d_i$, which can be used to modify $\bsig$ to a non-skew-symmetric form $\epsilon$, defined by $\epsilon(\psi_i,\psi_j)=\bsig(\psi_i,\psi_j)d_j=:\epsilon_{ij}$. We say that $\epsilon$ is \emph{skew-symmetrizable}.
 \end{itemize}
\end{Definition}

One of the main insights of \cite{GHK2} is that the seed data of Definition~\ref{def:seed} can be expressed in terms of the log Calabi--Yau geometry discussed above.

\begin{Construction} \label{constr:ghk-cluster-seed}
 Let $\frs$ be a seed as in Definition~$\ref{def:seed}$.
\begin{itemize}\itemsep=0pt
 \item The unfrozen vectors $\psi_1,\dots, \psi_r$ determine a fan of rays $\Sigma$ in $N$, which can be upgraded to a~stacky fan by specifying $d_i\psi_i$ as generators of the rays in $\Sigma$. This information determines a~toric stack $\oX_\Sigma$ with dense torus $\oT_{\CC} = N_{\CC^\times}$ and boundary divisors $D_i = (N/\psi_i)_{\CC^\times}\allowbreak\times B\ZZ/d_i$.
 \item For $i\leq r$, The pairing $\chi_i := \bsig(\psi_i,-)$ determines an element $\chi_i\in M$, which by skew-symmetry of $\bsig$ is contained in $\psi_i^\perp$ and hence can be treated as a character of the $(n-1)$-torus $D_i$. We write $H_i = \{\chi_i + 1=0\}$.
 \item The skew bilinear form $\bsig\in \wedge^2M \simeq H^0\big(\Omega_{\oX}^2\big(\log\oD\big)\big)$ determines a holomorphic $2$-form on~$\oX$ with simple poles along $\oD$. As explained in {\rm \cite[Lemma 3.4]{HK-icm}}, the definition of $\chi_i$ ensures that $\bsig$ extends to a form in $H^0\big(\Omega_X^2(\log D)\big)$.
\end{itemize}
	Hence we get a log Calabi--Yau pair $(X,D)$ as in Construction~$\ref{constr:ghk-cluster}$. We will call the variety $U= X-D$ the {\em truncated cluster variety} associated to the seed $\frs$.
\end{Construction}

\begin{Remark}
 We refer to $U$ as a ``truncated cluster variety'' to distinguish it from the more traditional notions of cluster variety. Fock and Goncharov originally defined cluster varieties as unions of cluster charts; these cluster varieties come in two types, $\mathcal{X}$ and $\mathcal{A}$. Cluster $\mathcal{A}$-varieties are quasi-affine; from our point of view, the correct variety to consider is the affine closure of the cluster $\mathcal{A}$-variety. %This affine closure is what we will call the GHK cluster variety.

The variety $U$ constructed above is, up to codimension 2, a cluster $\mathcal{A}$-variety, and it has the same affine closure. It is for this reason that we call a ``truncated cluster variety''. Throughout this paper, we will only ever work with the variety~$U$ as constructed above, so we will often drop the adjective ``truncated''.

 The simplest case of this phenomenon is treated in the following example.
\end{Remark}

\begin{Example}\label{ex:nontruncated}
 Let $\frs$ be the seed data corresponding to the variety $U=\CC^2\setminus \{xz=1\}$ from Example~\ref{ex:bside-ex1}, and write $\widetilde{\frs}$ for the product $\frs\times \frs$ of two copies of this seed data. The cluster variety associated to $\widetilde{\frs}$ ought to be
 \[
 \widetilde{U}:=U\times U = \CC^4\setminus \{(x_1z_1-1)(x_2z_2-1)=0\},
 \]
 but the blow-up-and-delete construction described above produces instead the space
 \[
 \widetilde{U} \setminus \{(0,0,0,0)\}.
 \]
\end{Example}

Finally, there is also a geometric interpretation of cluster mutation. Recall how cluster mutation alters seed data:

\begin{Definition}\label{def:mutation}
 Starting from a seed $\frs$ as in Definition~\ref{def:seed} and a choice of unfrozen basis vector $\psi_k$, the {\em mutation at $\psi_i$} is a new seed $\mu_i\frs$, produced by changing the basis $\{\psi_i\}$ to the basis $\{\psi_i'\}$ determined by the formula
 \[
 \psi_i' = \begin{cases}
 \psi_i + [\epsilon_{ik}]_+\psi_k & \text{for}\ i\neq k,
 \\
 -\psi_k & \text{for}\ i=k,
 \end{cases}
 \]
 where we define $[r]_+ = \max(0,r)$.
\end{Definition}

The geometry underlying mutations is described in \cite[Section 3.1]{GHK2}, which we can summarize as follows:
\begin{Proposition}[{\cite{GHK2}}]\label{prop:bside-mut}
 Let $(X,D)\to \big(\oX,\oD\big)$ be a log Calabi--Yau variety with toric model induced from a seed $\frs$ as in Construction~$\ref{constr:ghk-cluster-seed}$. Given a choice of an unfrozen vector $\psi_i$ in the seed, an elementary transformation of $\PP^1$-bundles induces a new toric model $(X',D')\to \big(\oX',\oD'\big)$, which corresponds under Construction~$\ref{constr:ghk-cluster-seed}$ to the mutated seed $\mu_i\frs$. The induced map of tori $\oT_\CC\to\oT_\CC'$ is the corresponding cluster transformation.
\end{Proposition}

The local model for this elementary transformation of $\PP^1$-bundles is described in the following example.
\begin{Example}\label{ex:elem-mut} Let
\[\oX = \PP^1_{[z_0:z_1]}\times \CC^\times_w\supset\oD = \{z_0z_1=0\}= (\{0\}\cup\{\infty\})\times \CC^\times,\]
and let $H = \{z_0=0, w + 1=0\}$, so that $\bigl(\oX,\oD\bigr)$ is a toric model for the log Calabi--Yau pair $\bigl(X = \Bl_H\oX, D=\widetilde{\oD}\bigr)$ discussed in Example \ref{ex:bside-ex1}. Let $\widetilde{H}$ be the strict transform of the line $\{w+1=0\}\subset \oX$. Then the blowdown of $\widetilde{H}$ can be used to define a new toric model
\[
f\colon\ (X,D)\to\big(\oX',\oD'\big)\cong \PP^1\times \CC^\times,
\]
where now the new blowup locus $H'\subset \oD'$ lives in $\{\infty\}\times \CC^\times$ rather than $\{0\}\times \CC^\times$. An SYZ image of this mutatation is illustrated below in Figure~\ref{fig:bsyz-mut}.
\end{Example}

\subsection{SYZ bases}\label{sec:bsyz}
In many cases, the symplectic mirror to $U$ which we will construct has traditionally been des\-cribed (at least in dimension 2)
by means of the toric base diagrams from \cite{Sym}, which can be used to express Lagrangian torus fibrations with singularities. In fact, the resulting torus fibrations will be SYZ fibrations for our mirror construction. We will not use this fact in an essential way, but for completeness of the exposition we briefly indicate here the corresponding SYZ fibration on the truncated cluster varieties.

Our construction on both sides of mirror symmetry begins with the usual Lagrangian torus fibration on $(\CC^\times)^n$ and then asks how the fibration is affected by a modification of this space. In~the present case, this means understanding how the torus fibration on the dense torus $\oT_\CC=\oX\setminus \oD$ is affected by the GHK blow-up-and-delete construction. As this construction is local to each divisor, we only need to understand it in a single local model, treated in the following example.

\begin{Example}[{\cite[Example 4.4]{HK-icm}}]
Consider again the basic local model described in Example~\ref{ex:bside-ex1}, where
the log Calabi--Yau pair $(X,D)$ is obtained from toric model $\big(\oX,\oD\big) = (\CC\times \CC^\times, \{0\}\times \CC^\times)$ by blowing up at center $H=\{(0,-1)\}$.

From the perspective of the moment-map base $\RR_{\geq 0}\times \RR$ of $\oX$, the blowup at $H$ has the effect of deleting a triangle with base on the $y$-axis centered at $(0,0)$ and identifying two of its edges, as illustrated in Figure~\ref{fig:syz-b1}.
 The manifold $U = X\setminus D$ then admits a fibration over the interior of this base, with generic fiber a torus and nodal (focus-focus) singularity above the vertex of the triangle.
 Away from the singular fiber, the character lattices of the torus fibers equip the base with the structure of an integral affine manifold, with nontrivial monodromy around the singular fiber given by a Dehn twist in the vanishing cycle.
 After taking a branch cut (necessary because of the monodromy of the integral affine structure), we can embed this integral affine manifold in $\RR^2$.

 \begin{figure}[!ht]\vspace{2ex}
 \centering
 \includegraphics[scale=0.2]{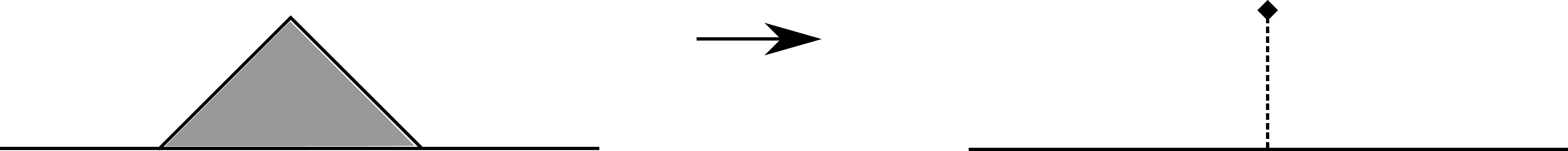}
 \caption{The process of producing the base for a torus fibration with singularity on the blowup $\Bl_{(0,1)}(\CC\times \CC^\times)$, by starting from the moment base of the original toric variety, then deleting a triangle and then identifying its edges. The dashed line indicates a branch cut in the integral affine structure.}
 \label{fig:syz-b1}
 \end{figure}
 \label{ex:syz-b-basic}
\end{Example}

This example makes clear the relation between the SYZ base of a cluster surface and the fan~$\Sigma$ in a toric model, since each blowup in Construction \ref{constr:ghk-cluster} has a local model as in Example~\ref{ex:syz-b-basic}, and these models only affect the SYZ base locally near the components of $H$, which do not interact.
We learned the following construction, and especially the diagram in Figure~\ref{fig:syz-b2}, from Andrew Hanlon:

\begin{Construction}
 Let $\Sigma$ be a fan of rays in $\RR^2$. The {\em SYZ base} $B$ associated to $\Sigma$ is the integral affine manifold which topologically is $\RR^2$, but with singularities in the integral affine structure as follows: for each ray $v=\left(\begin{smallmatrix}\psi_1\\\psi_2\end{smallmatrix}\right)$ in $\Sigma$, there is a singularity at a point $p_v$ on $v$, with monodromy matrix
 $\left(\begin{smallmatrix}1+\psi_1\psi_2&-\psi_1^2\\\psi_2^2&1-\psi_1\psi_2\end{smallmatrix}\right)$.
 After taking a branch cut at each $p_v$ pointing in direction~$v$, $B$ can be embedded in $\RR^2$ in the obvious way, as illustrated in Figure~$\ref{fig:syz-b2}$.
 \begin{figure}[!ht]\centering
	\includegraphics[scale=.135]{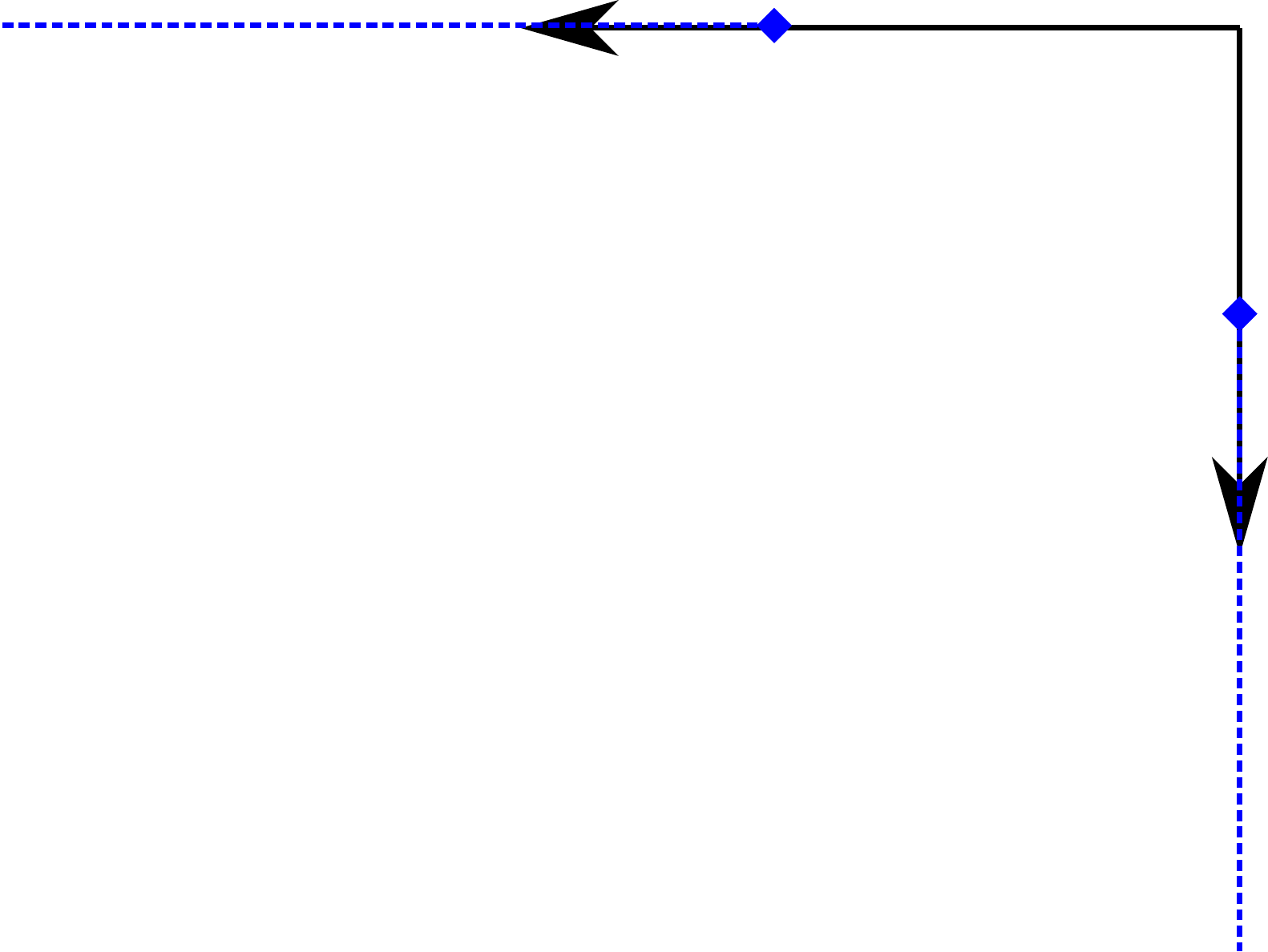}
 \caption{The fan $\Sigma$ for $\CC^2\setminus\{(0,0)\}$ superimposed on the SYZ base for the $A_2$ cluster variety, where the dashed lines emanating from singularities represent the branch cuts.}
 \label{fig:syz-b2}
 \end{figure}
 \label{constr:syz-b}
\end{Construction}

\begin{Proposition}Let $U$ be the cluster surface with toric model $\big(\oX_\Sigma,\oD\big)$. Then the integral affine structure on the base $B$ constructed in $\ref{constr:syz-b}$ is induced by a Lagrangian torus fibration $U\to B$ with focus-focus singularities above the singularities of $B$.
\end{Proposition}

\begin{proof}Let $\overline{\Delta}$ be the moment polytope of $\oX_\Sigma$, so that we can begin with the usual moment fibration $\oX_\Sigma\to \overline{\Delta}$. We would like to modify this fibration as we perform the GHK blow-up-and-delete construction. But from the above discussion, we see that such modifications can be performed locally on the base $\overline{\Delta}$, and that each such modification is given by the local model of
 Example \ref{ex:syz-b-basic}, modified by an integral affine-linear change of coordinates $A$ in $\SL(2,\ZZ)$. If
 $\left(\begin{smallmatrix}1\\0\end{smallmatrix}\right)$ is an eigenvector for the monodromy matrix of a singularity, and the transformation $A$ takes it to
 $\left(\begin{smallmatrix}\psi_1\\\psi_2\end{smallmatrix}\right)$,
 then $A$ must conjugate the monodromy matrix
 $\left(\begin{smallmatrix}1&1\\0&1\end{smallmatrix}\right)$ to
 $\left(\begin{smallmatrix}1+\psi_1\psi_2&-\psi_1^2\\\psi_2^2&1-\psi_1\psi_2\end{smallmatrix}\right)$.
\end{proof}

\begin{Remark}
 An analogue of the above proposition holds in higher dimensions, where now we replace the local model of Example \ref{ex:syz-b-basic} by its product with $(\CC^\times)^{n-2}$. As a result, the corresponding SYZ base will no longer have isolated singularities, but rather will have affine-linear subspaces of singularities (corresponding to the product of $\RR^{n-2}$ with the singularities from the 2-dimensional case). We do not go into detail about this case, since it is less useful for drawing pictures than the 2-dimensional case is.
 %but we revisit it on the mirror side in Section~\ref{sec:trades}. See in particular Figure~\ref{fig:trade-3d} for an illustration of the local model in dimension $n=3$.
\end{Remark}

One simple instance when the integral affine geometry described above can be helpful is in picturing the elementary transformation underlying the mutation construction of Proposition~\ref{prop:bside-mut}.
\begin{Example}
 Consider again the local model for an elementary transformation of $\PP^1$-bundles described in Example~\ref{ex:elem-mut}, where $X$ is obtained from $\oX$, $\oX'=\PP^1\times \CC^\times$ by non-toric blowup at a~single point in the toric boundary divisor $\{0\}\times\CC^\times$ (respectively $\{\infty\}\times \CC^\times)$.
 Two different SYZ bases for $X$ are depicted in Figure~\ref{fig:bsyz-mut}, depending on a choice of these two toric models. After deleting the divisor $D$ from $X$, the resulting SYZ bases for $U$ will be the same, differing only in a choice of branch cut used to embed them in $\RR^2$.
\begin{figure}[!ht]
\centering
 \includegraphics[scale=.165]{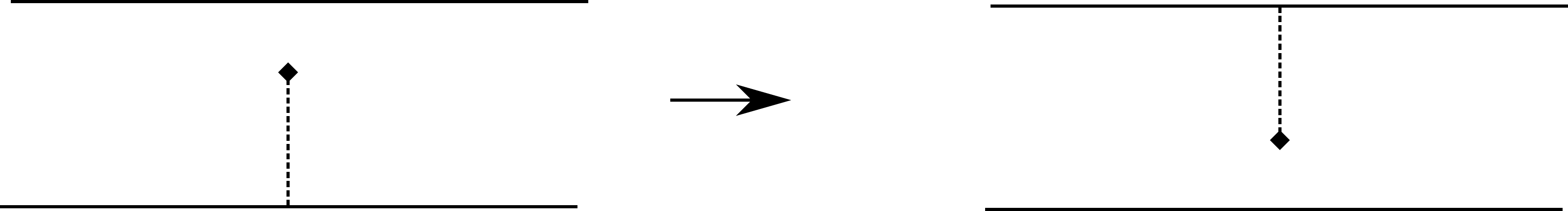}
 \caption{SYZ bases for $\Bl_{(0,-1)}\PP^1\times \CC^\times$ and $\Bl_{(\infty,-1)}\PP^1\times \CC^\times$, respectively.}
 \label{fig:bsyz-mut}
\end{figure}
\end{Example}

\section{Symplectic geometry background}\label{sec3}
\subsection{Liouville sectors}
We begin by recalling some basic facts about Weinstein manifolds and Liouville sectors from \cite{Eliashberg-revisited,GPS1}.

\begin{Definition}
 A {\em Liouville domain} is a compact symplectic manifold $(X,\omega)$ with boundary together with a choice of primitive $\omega = d\lambda$ such that $\lambda|_{\partial X}$ is a contact form, or equivalently such that the Liouville vector field $V = \omega^{-1}(\lambda)$ is outwardly transversal to the boundary. A Liouville domain $X$ can be completed to a {\em Liouville manifold} $\hat{X}$ by attaching a cylindrical end:
 \[
 \hat{X} = X\cup(\partial X\times [0,\infty)),
 \]
 where we extend $\lambda$ to $\hat{X}$ as ${\rm e}^s(\lambda|_{\partial X})$ at the end.
\end{Definition}

As Liouville domains and Liouville manifolds can be obtained from each other, we will not be careful about distinguishing them. Features of the boundary $\partial X$ (for instance, Weinstein hypersurfaces) can be equivalently described in terms of the ideal contact boundary of $\hat{X}$.

 \begin{Definition}
 The Liouville domain $(X,\omega,\lambda)$ is in addition {\em Weinstein} if it can be equipped with an exhausting Morse--Bott function $f\colon X\to \RR$ for which the Liouville vector field $V$ is gradient-like.
\end{Definition}
The Weinstein condition exists in order to rule out any pathological behavior of the Liouville vector field. All of the Liouville manifolds and sectors we consider in this paper will be Weinstein, and we will usually suppress the function $f$ in the notation. However, note that if $X$ is a Stein manifold with Morse--Bott K\"ahler potential $f$, then $f$ underlies a Weinstein structure with $\lambda = d^cf$.

Weinstein manifolds are so named in honor of the discovery, in \cite{Weinstein}, that they are a class of symplectic manifolds which can be glued together from simpler pieces. Today, we understand this gluing data, for a Weinstein manifold $X$, as specifying a cover of $X$ by {\em Liouville sectors.} These are Liouville manifolds with boundary which can be obtained from Liouville domains by completing along only a subset of the boundary.

\begin{Example}
	The manifold $T^*S^1=S^1\times \RR$ is a Liouville manifold obtained as a completion of the Liouville domain $S^1\times [-1,1]$. Specifying a finite number of points on the boundary of this Liouville domain determines a Liouville sector, by completing the boundary away from neighborhoods of those points.
\end{Example}

We will refer to \cite[Definition 1.1]{GPS1} for the technical definition of a Liouville sector, and we will content ourselves here with the simpler but ultimately equivalent notion from \cite{Eliashberg-revisited} of a~Weinstein pair.

\begin{Definition}
 A {\em Weinstein pair} is the data of a Weinstein domain $(X,\omega,\lambda)$ together with a~real hypersurface $\Sigma\subset \partial X$ such that $(\Sigma,d\lambda|_\Sigma,\lambda|_\Sigma)$ is itself a Weinstein manifold.
\end{Definition}

A Weinstein pair $(X,\Sigma)$ gives rise to a Liouville sector in the sense of \cite{GPS1} by completing~$X$ away from a standard neighborhood of $\Sigma$. Thus $\Sigma$ becomes the boundary of the resulting Liouville sector, along which we can perform gluings:

\begin{Construction}
 Let $(X_1,\Sigma_1)$ and $(X_2,\Sigma_2)$ be two Weinstein pairs, together with an isomorphism of Weinstein manifolds $\Sigma_1\cong \Sigma_2$. Then the Weinstein manifolds $X_1$ and $X_2$ can be glued together along a neighborhood of $\Sigma_i$ to produce a new Weinstein manifold $X_1\cup_\Sigma X_2$. If the boundary $\Sigma_i$ has several components, we can glue along one of them, and the resuling gluing will still be a Weinstein pair / Liouville sector.
\end{Construction}
\begin{Remark}
 Higher-codimension gluings are also possible, using the theory of Liouville sectors with (sectorial) corners from \cite{GPS2}, but will not be necessary in this paper.
\end{Remark}

All of the gluing data describing a Weinstein manifold $X$ can be encapsulated in a singular Lagrangian $\LL\subset X$, the skeleton of $X$.
\begin{Definition}
 For $X$ a Weinstein manifold, the {\em skeleton} of $X$ is the stable set
 \[
 \LL := \big\{x\in X\mid \varprojlim_{t\to \infty}\phi_t(x) \in X\big\}
 \]
 consisting of all points $x\in X$ whose image under the flow $\phi_t$ of the Liouville vector field eventually converges.

 For $X$ a Liouville sector, we allow $\LL$ to contain all those points which flow to the boundary of $X$.
	%[{\color{blue} When we talked it seemed like this definition wasn't quite right, or wasn't consistent with the definition of Liouville sector which you explained to me.}]
	In particular, for $(X,\Sigma)$ a Weinstein pair, the {\em skeleton} $\LL_{(X,\Sigma)}$ of the associated Liouville sector is the union
 \[\LL_{(X,\Sigma)} = \LL_X\cup \Cone(\Lambda),\]
 of the skeleton of $X$ with the cone, under the Liouville flow, of the skeleton $\Lambda$ of $\Sigma$.
\end{Definition}

Hence the skeleton of a Weinstein gluing $X_1\cup_\Sigma X_2$ is the gluing of the skeleta of the corresponding sectors along their glued boundary:
\[
\LL_{X_1\cup_\Sigma X_2} = \LL_{X_1}\cup_{\LL_{\Sigma}}\LL_{X_2}.
\]

The main source for Liouville sectors in mirror symmetry is via Landau--Ginzburg models.
\begin{Example}
 Let $X$ be a complex affine variety, and $\pi\colon X\to \CC$ a global holomorphic function. We can treat a general fiber $\pi^{-1}(1)$ as a Weinstein hypersurface inside the contact boun\-dary~$\partial^\infty X$ and hence $\big(X,\pi^{-1}(1)\big)$ as a Weinstein pair. We will often denote the associated Liouville sector by $(X,\pi)$.
\end{Example}
\subsection{Fukaya categories and descent}
In \cite{GPS1}, a wrapped Fukaya category $\cW(X)$ is associated to any Liouville sector $X$. In the case where $X$ is a Liouville manifold, this is the usual wrapped Fukaya category of $X$, and in the case where $X$ is a Landau--Ginzburg sector $(X,\pi)$, this is the Fukaya--Seidel-type category associated to the fibration $\pi$.
We refer to \cite{GPS2,GPS3,GPS1} for details of this construction, and we summarize here its most important properties.

\begin{Proposition}[\cite{GPS1}]
 The wrapped Fukaya category is covariantly functorial for inclusions of Liouville sectors: a Liouville subsector $i\colon X\subset Y$ induces a functor $i\colon \cW(X)\to \cW(Y)$.
\end{Proposition}

The central result of the series \cite{GPS2,GPS3,GPS1} is that not only do Liouville sectors glue along their skeleta at the boundary, but their Fukaya categories glue in precisely the same way:

\begin{Theorem}[\cite{GPS2}]\label{thm:codescent}
 The wrapped Fukaya category satisfies descent for Weinstein sectorial covers: Let $U_1,\dots,U_n\subset X$ be Liouville subsectors of $X$ which cover $X$ whose intersections $U_I:=\bigcap_{i\in I}U_i$, for $I\subset[n]$, are also Liouville sectors, and let $\calP$ be the poset of inclusions of the subsectors $U_I$. Then the natural map
 \[
 	\colim_{I\in \calP}\cW(U_I)\to \cW(X)
 \]
 from the homotopy colimit of the wrapped Fukaya categories of the $U_I$ to the wrapped Fukaya category of $X$ is an equivalence of dg categories.
\end{Theorem}

In other words, if we can present a Weinstein manifold $X$ as a gluing of various Liouville sectors $U_i$ whose Fukaya categories (and corestriction functors) we understand, then we can compute the Fukaya category of $X$ itself. One trick which is occasionally helpful in such computations is the following (described for instance at \cite[Lemma 1.3.3]{Gaits-DG}): if the functors in a colimit diagram are all continuous, and the categories are all cocomplete, then the colimit is equivalent to the limit of the opposite diagram obtained by passing from all functors to their right adjoints.

\begin{Corollary}\label{cor:descent}
 Let $U_i$ and $X$ be as in Theorem~$\ref{thm:codescent}$, and write $\cW^\infty$ for the ind-completion of the wrapped Fukaya category. Then the natural map
 \[
 \cW^\infty(X)\to\varprojlim_{I\in \calP^{\rm op}}\cW^\infty(U_I)
 \]
 is an equivalence. The wrapped Fukaya category $\cW(X)$ can therefore be recovered as the category of compact objects inside the homotopy limit $\varprojlim_{I\in \calP^{\rm op}}\cW^\infty(U_I)$.
\end{Corollary}
This is particularly useful in case each Fukaya category $\cW(U_I)$ is equivalent to the category $\Coh\big(U_I^\vee\big)$ of coherent sheaves on a {\em smooth} algebraic variety $U_I^\vee$, in which case the ind-completed category $\cW^\infty(U_I)\cong \Ind\Coh\big(U_I^\vee\big)$ will be equivalent to the category $\QCoh(U_I)$ of quasicoherent sheaves on~$U_I$.

The computations of the Fukaya categories $\cW(U_I)$ of individual subsectors can be accomplished in the case when the subsectors $U_I$ are cotangent bundles (with sectorial structure); in this case, the Fukaya category can be computed as a category of constructible sheaves on the base, or, more invariantly, microlocal sheaves along the Lagrangian skeleton:

\begin{Theorem}[{\cite[Theorem 1.1]{GPS3}}]%\label{thm:NZ}
 Let $(T^*M, \Sigma)$ be a Liouville sector, and $\Lambda\subset S^*M$ the skeleton of $\Sigma$ inside the contact boundary of $T^*M$. Write $\Sh^\infty_{-\Lambda}(M)$ for the cocomplete dg category of $($possibly infinite-dimensional$)$ sheaves on $M$ whose singular support away from the zero section is contained in $-\Lambda$. Then there is an equivalence of categories
 \[
 \Sh^\infty_{-\Lambda}(M) \cong \cW^\infty(T^*M,\Sigma)
 \]
 between this category of sheaves and the Fukaya category of the Liouville sector associated to $(T^*M,\Sigma)$.

 If we write $\wsh_{-\Lambda}(M)$ $($for ``wrapped constructible sheaves''$)$ for the category of compact objects in $\Sh^\infty_{-\Lambda}(M)$, then we obtain an equivalence
 \[
 \wsh_{-\Lambda}(M) \cong \cW(T^*M,\Sigma).
 \]
\end{Theorem}
\begin{Remark}
 Note that the category $\wsh_{-\Lambda}(M)$ is not in general equivalent to the usual category of constructible sheaves on $M$ with singular support in $-\Lambda$, since it includes some objects whose stalks are not finite-dimensional. For instance, if $M=S^1$ and $\Lambda$ is empty, then $\wsh_{-\Lambda}(M)$ includes the universal local system $k[x^\pm]$ on $S^1$.
\end{Remark}
\begin{Remark}
	The minus sign in the theorem, denoting the anti-symplectic involution $(-)$: $T^*M\to T^*M$ which acts on the fibers as scaling by $(-1)$, appears by convention in order to fix the fact that $\Sh_{\Lambda}^\infty(M)$ is actually isomorphic not to the Fukaya category $\cW^\infty(T^*M,\Sigma)$ but to its {\em opposite} category; as the map $(-)$ is anti-symplectic, it acts at the level of categories as an anti-involution, cancelling out this opposition.
\end{Remark}

We will not need to perform any serious constructible-sheaf calculations in this paper: we rely on the computation in \cite{Ku}, summarized in Theorem~\ref{thm:ccc} below. The difficulties for us will consist only in gluing these calculations together.

\begin{Remark}
 In case a Weinstein manifold $X$ has a cover by sectors of the form $(T^*M,\Sigma)$, the above theorems show that there exists a cosheaf of categories on the skeleton $\LL_X$ of $X$, locally isomorphic to $\wsh_{-\Lambda}(M)$, whose homotopy colimit recovers the wrapped Fukaya category of $X$. We denote this cosheaf by $\wmsh$, and call it the cosheaf of {\em wrapped microlocal sheaves} along $\LL_X$. In fact, this cosheaf has now beeen constructed in \cite{NS20} on the skeleton of a general Weinstein manifold.
\end{Remark}

For our discussion of mutations in Section \ref{sec:mut} below, we will find it helpful to note one other functoriality of the wrapped Fukaya category. In addition to its covariant functoriality under embeddings of Liouville sectors, the wrapped Fukaya category admits a contravariant functoriality under embeddings of Liouville {\em domains}, which is expected to agree with the {\em Viterbo functoriality} of \cite{AS-viterbo}.

\begin{Proposition}[{\cite[Section 8.3]{GPS2}}]\label{prop:viterbo}
 Let $Y$ be a Liouville sector, and $X\subset Y$ a Liouville subdomain. Then there is a restriction functor $\cW(Y)\to\cW(X)$.
\end{Proposition}

\subsection{Toric mirror symmetry}
We now specialize to the class of Liouville sectors most relevant to our discussion: those arising from Landau--Ginzburg mirrors to toric varieties.
\begin{Definition}
 Let $\oX_\Sigma$ be a toric variety with fan $\Sigma$ of cones in $N_\RR$. Let $\Delta\subset N_\RR$ be the convex hull of the generators of the rays in $\Sigma$, and let $W_\Sigma\colon M_{\CC^\times}\to \CC$ be a Laurent polynomial with Newton polytope $\Delta$. The Landau--Ginzburg model $(M_{\CC^\times},W)$ is a {\em Hori--Vafa mirror} to~$\oX_\Sigma$.
\end{Definition}

The skeleton of the Liouville sector $(M_{\CC^\times},W)$ (though not originally recognized as such) appeared implicitly in work of Bondal \cite{Bo-ccc} and was studied extensively by Fang--Liu--Treumann--Zaslow \cite{FLTZ2,FLTZ1,FLTZ3}.

\begin{Definition}\label{defn:bondal-lag}
 Let $\Sigma$ be a fan of cones in $N_\RR$, and note that we can treat $N_\RR$ as the cotangent fiber of the cotangent bundle $T^*M_{S^1} = M_\RR/M\times N_\RR$.
 The (negative) {\em Bondal Lagrangian} is a~conic Lagrangian $\LL_\Sigma\subset T^*M_{S^1}$ defined by
 \begin{equation}\label{eq:bondal-lag}
 \LL_\Sigma = \bigcup_{\sigma \in \Sigma}\overline{\sigma^\perp}\times \sigma\subset M_\RR/M\times N_\RR,
\end{equation}
where we write $\overline{\sigma^\perp}$ for the image of the linear subspace $\sigma^\perp\subset M_\RR$ in the quotient $M_\RR/M$.
\end{Definition}

\begin{Example}
 Let $\dim n=2$ and $\Sigma$ be the fan whose only nonzero cone is the ray $\RR\langle e_1\rangle$. Then the Lagrangian $\LL$ is a torus with a cylinder attached, as in Figure~\ref{fig:bondal-example}.
 The torus corresponds in the decomposition of equation~\eqref{eq:bondal-lag} to the zero cone in $\Sigma$, while the cylinder corresponds to the cone $\RR\langle e_1\rangle$.
 \begin{figure}[!ht]
 \centering
 \includegraphics[width=5cm]{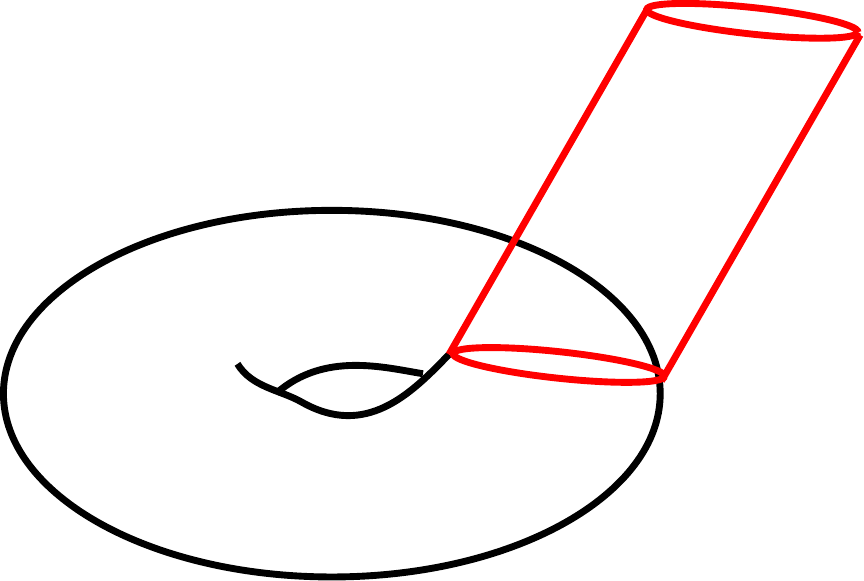}
 \caption{The skeleton $\LL_\Sigma$ for $\Sigma$ the fan of $\CC\times\CC^\times$, given as the union of a torus (in black) with a~cylinder (in red).}
 \label{fig:bondal-example}
 \end{figure}
 \end{Example}

As mentioned in Remark \ref{rem:stackiness}, in the case of a cluster seed whose inner-product $\overline{\sigma}$ is skew-symmetric but whose associated form $\epsilon$ is only skew-symmetrizable (i.e., when not all the integers~$d_i$ are equal to 1), we need to scale some of the generators of rays in $\Sigma$, making $\oX_\Sigma$ into a~toric stack. In this case, we need to make a slight modification of the above definition. For $\sigma$ a~cone in $\Sigma$, write $\sigma_\ZZ:=\sigma\cap N$ for the sublattice defined by $\sigma$, and note that the tori $\overline{\sigma}^\perp$ are the Pontrjagin duals of the quotient lattices $N/\sigma_\ZZ$:
\begin{equation}\label{eq:sigmaperp}
\overline{\sigma}^\perp = \Hom(N/\sigma_\ZZ, \RR/\ZZ).
\end{equation}

Suppose now that $\Sigma$ is a stacky fan, so that
 the rays in $\Sigma$ are equipped with not necessarily primitive generators. If $\sigma$ is a cone on rays equipped with not necessarily primitive generators $\beta_1,\dots,\beta_k$, write $\sigma_\ZZ$ for the sublattice $\langle \beta_1,\dots,\beta_k\rangle\subset N$.
 In analogy with equation~\eqref{eq:sigmaperp}, we~define a not necessarily connected torus $G_\sigma$ as the Pontrjagin dual
 \[
 G_\sigma = \Hom(N/\sigma_\ZZ,\RR/\ZZ).
 \]
 The group $\pi_0(G_\sigma)$ of connected components of $G_\sigma$ takes into account the isotropy along the toric stratum corresponding to cone $\sigma$. The definition of the Lagrangian in the stacky case proceeds precisely in analogy to Definition~\ref{defn:bondal-lag}:
\begin{Definition}
 Let $\Sigma$ be a stacky fan of cones $\sigma$, and define the groups $G_\sigma$ as above.
 Then the {\em stacky Bondal Lagrangian} is defined as
 \[
 \LL_\Sigma = \bigcup_{\sigma \in \Sigma}G_\sigma\times \sigma\subset M_\RR/M\times N_\RR.
 \]
\end{Definition}

The motivation for studying the Lagrangian $\LL_\Sigma$ was the following result, studied by Bondal and FLTZ and eventually proved in full generality by Kuwagaki (later reproved in many cases in \cite{Zhou-ccc}):
\begin{Theorem}[\cite{Ku}]\label{thm:ccc}
 There is an equivalence of dg categories $\Sh^w_{-\LL_\Sigma}(M_{S^1})\cong\Coh\big(\oX_\Sigma\big)$ between the category of wrapped constructible sheaves on the torus $M_{S^1}$ with singular support in $-\LL_\Sigma$ and the category of coherent sheaves on the toric stack~$\oX_\Sigma$.
\end{Theorem}

The theory of Liouville sectors has made it possible to relate the Lagrangian $\LL_\Sigma$ to the Hori--Vafa Landau--Ginzburg model and hence to reframe the above in more traditional mirror-symmetric language.
\begin{Theorem}[\cite{GS17,Zhou-skel}]\label{thm:hv-skel}
 Let $\Delta\subset N_\RR$ be a polytope containing 0, and let $\Sigma$ be the fan of cones on the faces of $\Delta$. As above, let $W$ be a Laurent polynomial $M_{\CC^\times}\to \CC$ with Newton polytope $\Delta$.
 Then $\LL_\Sigma$ is the Lagrangian skeleton of the Liouville sector associated to $(M_{\CC^\times},W)$.
\end{Theorem}

{\sloppy\begin{Corollary}
 The dg category $\Sh^w_{-\LL_\Sigma}(M_{S^1})$ is equivalent to the wrapped Fukaya category $\cW(M_{\CC^\times}, W)$ of the Liouville sector associated to the Hori--Vafa Landau--Ginzburg model $(M_{\CC^\times}, W)$.
\end{Corollary}}

However, Theorem~\ref{thm:hv-skel} as stated does not apply to our situation, since the hypothesis of the theorem requires that $\Sigma$ contain higher-dimensional cones, corresponding to the positive-dimensional faces of $\Delta$, whereas we are interested in fans which contain only one-dimensional cones.

Nevertheless, in the case of a purely 1-dimensional fan $\Sigma$, the Lagrangian $\LL_\Sigma$ is still the skeleton of a perfectly good Liouville sector -- indeed, its boundary at infinity is a disjoint union of $(n{-}1)$-tori, corresponding to the structure of a Weinstein pair of the form $\big((\CC^\times)^n,\bigsqcup_{i=1}^r(\CC^\times)^{n-1}\big)$ --~and Theorem~\ref{thm:ccc} remains true in this case, so that we must continue to understand this sector as the mirror to the (no longer complete) toric variety $\oX_\Sigma$. We are therefore justified in the following notation:
\begin{Definition}\label{defn:mirrorsec}
We write $\oX_\Sigma^\vee$ for the Liouville sector with Lagrangian skeleton $\LL_\Sigma$.
\end{Definition}

\section{The mirror construction}\label{sec4}
\subsection{Weinstein handle attachment}\label{sec:surg}
The manifolds in our mirror construction will be obtained from the cotangent bundle of a torus by a simple Weinstein handle attachment. We begin with the two-dimensional case, studied extensively in \cite{STW}.

\begin{Construction}\label{constr:2d-surgeries}
 Let $T=T^2$ be a $2$-torus and $\psi_1,\dots,\psi_r\colon T\to S^1$ characters of $T$, defining codimension-$1$ subtori $S_1,\dots,S_r$. A choice of coorientations of the subtori $S_i$ determines lifts of the $S_i$ to Legendrian circles $\calL_i$ in the contact boundary $S^*T$ of $T^*T$. We write $W$ for the Weinstein manifold obtained by attaching one Weinstein handle along each Legendrian $\calL_i$.
\end{Construction}

\begin{Example}\label{ex:basic-surgery}
 Let $r=1$ and let $\psi=\psi_1$ be the character given in coordinates by $(\theta_1,\theta_2)\mapsto \theta_1$. Then the Weinstein manifold $W$ has skeleton given by the union of the torus $T$ and a Lagrangian disk glued to $T$ along $S=\{\psi=1\}$, as depicted two different ways in Figure~\ref{fig:ex-basic}.
\end{Example}
\begin{figure}[!ht]
 \centering
 \includegraphics[width=8cm]{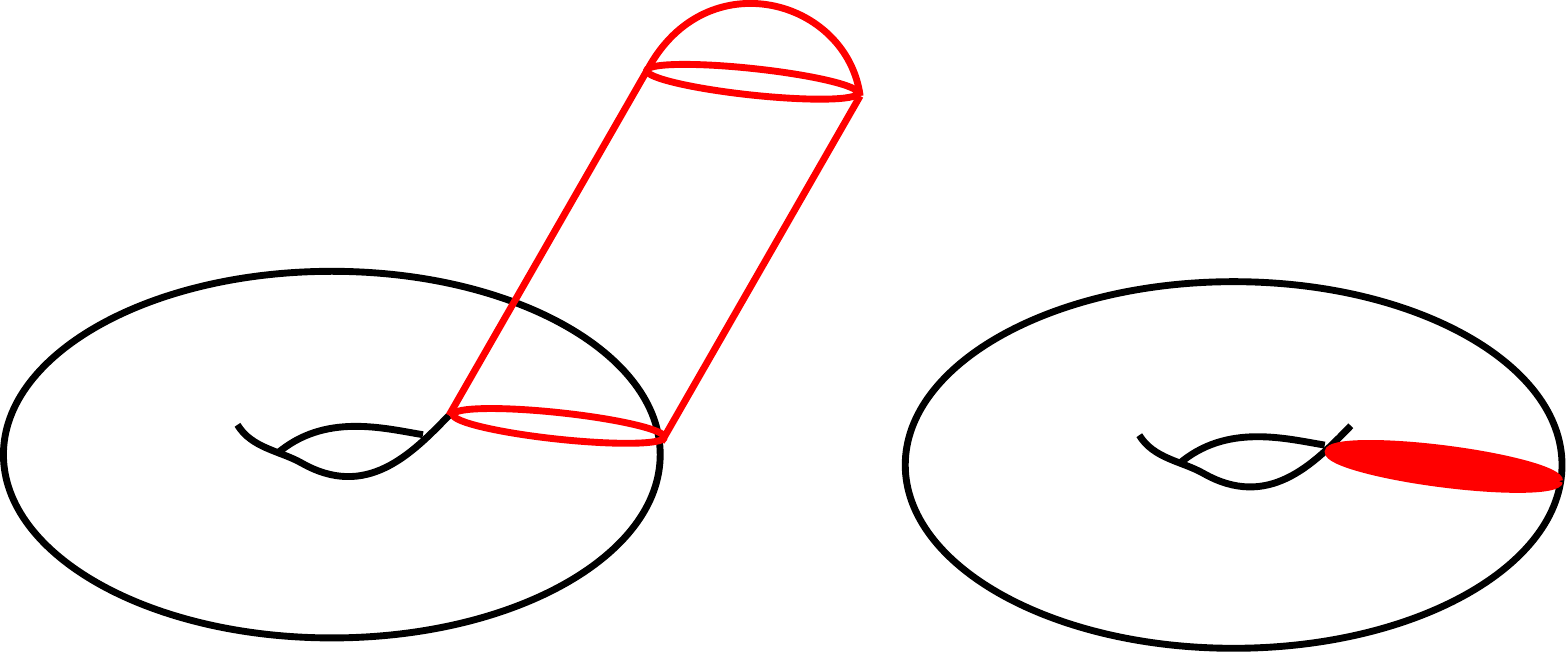}
 \caption{Two views of the skeleton obtained from a single Weinstein disk attachment, either as a torus with cylinder (the cone on Legendrian $\calL$) attached and then capped off, or as a torus with a disk glued in directly.}
 \label{fig:ex-basic}
\end{figure}
The Weinstein manifold $W$ constructed above admits a simple Liouville-sectorial cover.
Let $M=\ZZ^2$, $T=M_{S^1}$,
and note that each character $\psi_i$ is a vector in the dual vector space $N_\RR$; hence we can think of the collection of characters $\psi_i$ as a purely one-dimensional fan $\Sigma$ in $N_\RR$. Now the following is an immediate consequence of Definition \ref{defn:mirrorsec}:

\begin{Lemma}
 The Liouville sector determined by the pair $\big(T^*M, \bigsqcup \calL_i\big)$ is the Liouville sector~$\oX_\Sigma^\vee$ described in Definition~$\ref{defn:mirrorsec}$, the mirror to the $($incomplete$)$ toric variety $\oX_\Sigma$.
 %\label{lem:2dmirsec}
\end{Lemma}

The other piece in the Liouville-sectorial cover we consider will be a disjoint union of copies of $T^*D$, for $D$ a disk; this is the Weinstein manifold whose skeleton is the disjoint union of copies of the disk $D$. Note that for $A\subset D$ an annulus, the Liouville sector $T^*D$ admits $T^*A$ as a subsector.

\begin{Proposition}\label{prop:cover-2d}
 Let $W$ and $\Sigma$ as above. Then $W$ admits a cover by the Liouville sectors $\oX_\Sigma^\vee$ and $\bigsqcup_{i=1}^r T^*D$, which intersect in $\bigsqcup_{i=1}^r T^*A$.
\end{Proposition}

\begin{proof}
 This sectorial cover is nothing more than the expression of $W$ as a Weinstein disk attachment. See Figure~\ref{fig:ex-sectcover} for an illustration in the case $r=2$.
\end{proof}

\begin{figure}[!ht]
\centering
 \includegraphics[width=5cm]{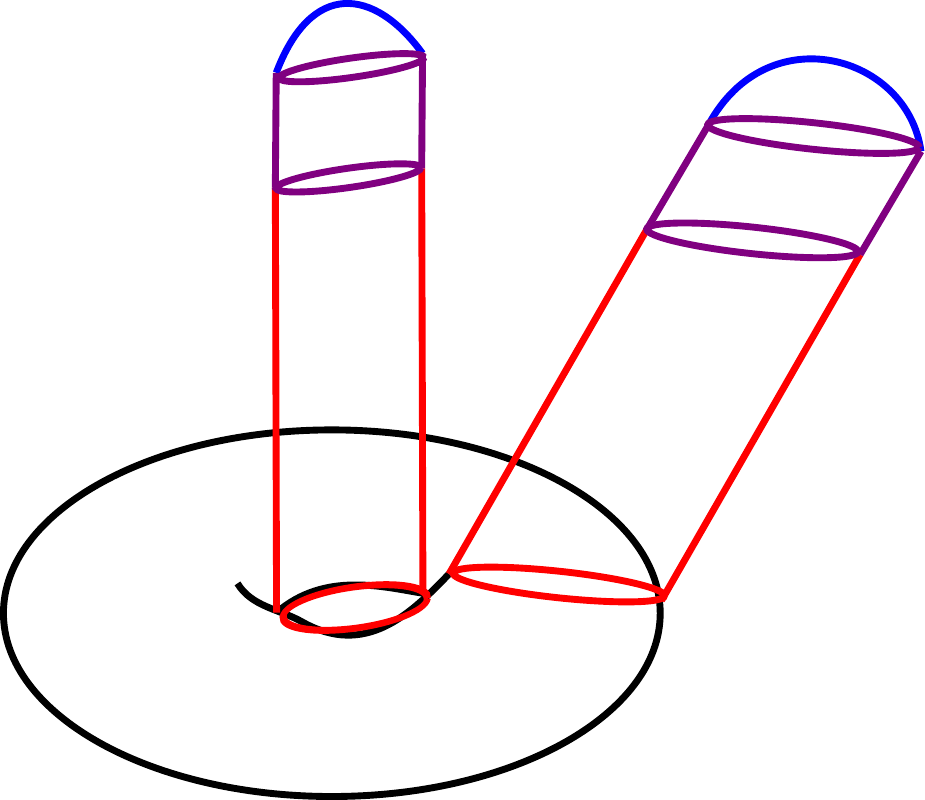}
 \caption{An image of the Liouville-sectorial cover of the Weinstein manifold obtained by gluing in disks along circles in homology classes $(1,0)$ and $(0,1)$. The three pieces in the cover are the torus with cylinders attached; a pair of disks; and their intersection, a pair of cylinders.}
 \label{fig:ex-sectcover}
\end{figure}
So far, we have only recalled the constructions already appearing in \cite{STW}. We now generalize to higher dimensions by taking a product of the local model for the above disk gluing with $T^*T^{n-2}=(\CC^\times)^{n-2}$.

\begin{Example}\label{ex:basic-surgery-highd}
 Let $W$ be the product of $(\CC^\times)^{n-2}$ with the Weinstein manifold studied in Example~\ref{ex:basic-surgery}. Then $W$ can be obtained as follows: starting with an $n$-torus $T=T^n$, define a~codimension-1 subtorus $S$ by the character $\psi\colon T\to S^1$ given by $(\theta_1,\dots,\theta_n)\mapsto \theta_1$, and use a~choice of coorientation to lift $S$ to a Legendrian $\calL\subset S^*T$, where now $\calL$ is an $(n-1)$-torus. Then glue $T^*T$ to $T^*\big(D\times T^{n-2}\big)$ along the Legendrian $\calL$.
\end{Example}

Note that unlike in the 2-dimensional case, our construction in higher dimensions will not be determined by the choice of Legendrian $\calL$: in addition, we need to know which of the remaining $n-1$ directions contains the disk capping off $\calL$. This will be a direction $\chi$ in the torus $S$.

\begin{Definition}
 Let $\chi\colon S^1\to T^{n-1}$ be a cocharacter of an $(n-1)$-torus $T^{n-1}$, so that $T^{n-1}$ can be expressed as a product $\big(S^1\big)_\chi\times T^{n-2}$. We write $\DD_\chi\times T^{n-2}$ for the result of replacing the first factor in this product with a disk.
\end{Definition}

We are thus ready to combine several copies of the move described in Example~\ref{ex:basic-surgery-highd}.

\begin{Construction}\label{cons:mirror}
 Let $T$ be the $n$-torus $M_{S^1}$, and let $\Sigma$ be a purely $1$-dimensional fan in~$N_\RR$, defining a collection of cooriented codimension-1 subtori $S_i\subset T$, lifting to Legendrians $\calL_i\subset S^*T$. Moreover, for each $S_i$ choose a cocharacter $\chi_i\colon S^1 \to S_i$ of the torus $S_i$. Choose a decomposition $S_i \simeq S^1_{\chi_i} \times T_i$, where $T_i$ is an $(n-2)$-torus. Finally, glue a copy of $T^*(D_{\chi_i}\times T_i)$ along each Legendrian $\calL_i$ as in Example~$\ref{ex:basic-surgery-highd}$. We denote the resulting Weinstein manifold by $U^\vee$.
\end{Construction}

Note that the resulting Weinstein manifold does not depend on the decomposition $S_i = S^1_{\chi_i} \times T_i$. The key point here is that given two decompositions $S_i \simeq S^1_{\chi_i} \times T_i \simeq S^1_{\chi_i} \times T_i'$, there is a canonical isomorphism $D_{\chi_i}\times T_i \simeq D_{\chi_i} \times T_i'$.

As in Proposition~\ref{prop:cover-2d}, this gluing presentation of the manifold $U^\vee$ immediately gives a~Liou\-ville-sectorial cover.

\begin{Proposition}\label{prop:cover}
 The Weinstein manifold $U^\vee$ defined in Construction~$\ref{cons:mirror}$ is a union of the Liouville sectors $\oX_\Sigma^\vee$ and $\bigsqcup_{i=1}^r T^*(D_{\chi_i}\times T_i)$, glued along $\bigsqcup_{i=1}^r T^*\big(S^1_{\chi_i}\times T_i\big)$.
\end{Proposition}

\begin{proof}
 As in Proposition~\ref{prop:cover-2d}, this follows immediately from the construction of $U^\vee$: one can obtain $U^\vee$ by beginning with sector $\oX_\Sigma^\vee$ and performing gluings along the $r$ disjoint Legendrians~$\calL_i$ as indicated.
\end{proof}
Combined with the codescent prescription of Theorem~\ref{thm:codescent}, this gives a presentation of the wrapped Fukaya category of $U^\vee$.

\begin{Corollary}\label{cor:a-pushout}
 The wrapped Fukaya category $\cW\big(U^\vee\big)$ of the Weinstein manifold defined by Construction~$\ref{cons:mirror}$ can be presented as a homotopy pushout of the diagram
 \begin{equation}\label{eq:a-pushout}
 \xymatrix{
 \cW\big(\bigsqcup_{i=1}^r T^*\big(S^1_{\chi_i}\times T_i\big)\big)\ar[r]\ar[d]&\cW\big(\bigsqcup_{i=1}^r T^*D_{\chi_i}\times T_i\big)\\
 \cW\big(\oX_\Sigma^\vee\big),
 }
\end{equation}
where the maps are the covariant functorialities of the wrapped Fukaya category under inclusions of Liouville sectors.
\end{Corollary}

Since the categories and maps in the pushout diagram \eqref{eq:a-pushout} all have well-understood interpretations in mirror symmetry, we can rephrase Corollary~\ref{cor:a-pushout} entirely in terms of the mirror algebraic geometry used by Gross--Hacking--Keel in Construction \ref{constr:ghk-cluster} .

Recall that construction also began with the construction of a fan $\Sigma\subset N_\RR$ of rays $\langle d_i\psi_i\rangle$, defining a toric variety $\oX_\Sigma$ with boundary divisor
\[
\oD=\bigsqcup_{i=1}^r \oD_i\cong \bigsqcup_{i=1}^r(N/\psi_i)_{\CC^\times}\times B\ZZ/d_i,
\]
along with characters $\chi_i$ of these quotient tori,
defining a subvariety
\[
H\cong \bigsqcup_{i=1}^r (\CC^\times)^{n-2}\times B\ZZ/d_i\subset \oD.
\]
We write $i\colon\oD\to \oX_\Sigma$ for the inclusion of the toric boundary divisor and $j\colon H\to \oD$ for the inclusion of $H$ into $\oD$. Now we are ready to restate Corollary~\ref{cor:a-pushout}:
\begin{Corollary}\label{cor:b-pushout}
 The wrapped Fukaya category $\cW\big(U^\vee\big)$ of $U^\vee$ can be presented as a homotopy pushout of the diagram
 \begin{equation}
 \xymatrix{
 \Coh\big(\oD\big)\ar^-{j^*}[r]\ar_-{i_*}[d]&\Coh(H)\\
 \Coh\big(\oX_\Sigma\big).}
 \label{eq:b-pushout}
 \end{equation}
\end{Corollary}
\begin{proof}
 The equivalences between the categories in diagram \eqref{eq:b-pushout} and those in diagram \eqref{eq:a-pushout} have already been discussed. \big(For the stacky case, where not all $d_i$ are equal to~1, note that there is an equivalence of dg categories $\Coh(B\ZZ/d_i)\cong \Coh\bigl( \bigsqcup_{j=1}^{d_i}\Spec(k) \bigr)$.\big) That the vertical arrows agree is proved in \cite[Section~7.2]{GS17} (technically the result there is for the pullback $i^*$; the result for the pushforward can be recovered by passing to right adjoints), and the agreement of the horizontal arrows is clear.
\end{proof}
The above presentation is very pleasant to have, though it may at first seem perplexing and indeed to an algebraic geometer it will likely not look ``geometric'', due to the opposing functorialities present. Nevertheless, we will see in our main result, Theorem \ref{thm:main} below, that this pushout does indeed present the category of coherent sheaves on the truncated cluster variety $U=X\setminus D$.

Before proceeding to that proof, we consider the most basic example.

\begin{Example}\label{ex:fund-calc}
 Consider again the situation of Example~\ref{ex:basic-surgery}, where $n=2$, and $r=1$, so that~$\chi$ is the character of $T=\RR^2/\ZZ^2$ given by $(\theta_1,\theta_2)\mapsto \theta_1$. In this case, the diagram \eqref{eq:b-pushout} becomes
 \[
 \xymatrix{
 \Coh(\{0\}\times \CC^\times)\ar[r]\ar[d]&\Coh\left(\{(0,-1)\}\right)\\
 \Coh(\CC\times \CC^\times),}\]
 where the vertical and horizontal maps are the evident pushforward and pullback functors, respectively.

 Let us rewrite this diagram by first passing to ind-completions and right adjoints as in Corollary \ref{cor:descent}. At the level of categories, this amounts to replacing coherent by quasicoherent sheaves (using the fact that all varieties involved are smooth -- otherwise, we would have to use ind-coherent sheaves), so that we obtain the pullback diagram
 \begin{equation}\label{eq:ex-pullback}
 \begin{split}&
 \xymatrix{
	k[x,y^\pm]/(x)\mmod&\ar[l] k[x,y^\pm]/(x,y+1)\mmod\\
	k[x,y^\pm]\mmod\ar[u],
	}\end{split}\end{equation}
	of module categories. The horizontal functor, right adjoint to pullback, is pushforward, (which on module categories is restriction of scalars); the vertical functor, as the right adjoint to pushforward, is the shriek pullback functor, which on a $k[x,y^\pm]$-module $M$ is the coextension of scalars
	\[
		M\mapsto \Hom_{k[x,y^\pm]}\big(k[x,y^\pm]/(x),M\big).
	\]
	By resolving $k[x,y^\pm]/(x)$ as a $k[x,y^\pm]$-module, we can also understand this as the functor taking a module $M$ to the cone of the map $x\colon M\to M$.

 Let us denote by $\cC$ the homotopy limit of diagram \eqref{eq:ex-pullback}.
 An object of $\cC$ is a triple $(M,N,\phi)$ consisting of a $k[x,y^\pm]$-module $M$, a $k[x,y^\pm]/(x,y+1)$-module $N$, and an isomorphism
 \[
\phi\colon\ \Cone\big(M\stackrel{x}\longrightarrow M\big)\stackrel{\sim} \longrightarrow N
 \]
	of $k[x,y^\pm]/(x)$-modules (regarding $N$ as one by restriction). The isomorphism $\phi$ can be reformulated as the imposition on $\Cone(x\colon M\to M)$ of the equation that $y+1$ acts as zero, or in other words, it is the data of a degree-$(-1)$ endomorphism $x'$ of this complex with $d(x') = y+1$.

	As explained in detail in \cite[Proposition 4.10]{STWZ}, if $M$ was concentrated in degree 0, then $x'$ is equivalent to the datum of a (degree-0) endomorphism of $M$ satisfying $xx'=y+1$: pictorially, we have
	\[
		\xymatrix{
			M\ar[r]^-x&M\ar[ld]_-{x'}\\
			M\ar[r]^-x&M,
			}
		\]
		where the horizontal map $x$ is the differential on $\Cone(x)$, so that $xx'=y+1$ expresses the triviality of $y+1$ acting on $\Cone(x)$.

	As expected, this agrees with the category $\QCoh(U) = k[x,x',y^\pm]/(xx'=y+1)\mmod$ of quasicoherent sheaves
	on the mirror cluster variety $U=\CC^2\setminus \{xx'=1\}$.
\end{Example}

 \subsection{Homological mirror symmetry}
Let $(X,D)$ be an $n$-dimensional log Calabi--Yau pair with toric model $\big(\oX,\oD\big)$, modeling the log Calabi--Yau variety $U=X\setminus D$ as in Definition~\ref{def:toricmodel}.
Recall that such a pair is determined by the following two pieces of data:
\begin{itemize}\itemsep=0pt
 \item A purely one-dimensional fan $\Sigma$ in $N_\RR$, the union of rays $\langle d_1\psi_1\rangle, \dots, \langle d_r\psi_r \rangle$, where $\psi_i\in N$ are primitive lattice vectors, and $d_i\in \ZZ_{>0}$. As is usual in toric geometry, the $\psi_i$ can be understood as cocharacters of the torus $N_{\CC^\times}$.
 \item For each $1\leq i\leq r$, an element $\chi_i\in \psi_i^\perp= M/\psi_i$, which we think of as a character on the~divisor $D_i$ (thought of as a complex $(n-1)$-torus), defining a hypersurface $H_i=\allowbreak\{\chi_i=-1\}\subset D_i$.
\end{itemize}
This is the same data used in Construction \ref{cons:mirror} to define the Weinstein manifold $W^\vee$, where now we think of $\psi_i\in N$ not as a cocharacter of $N_{\CC^\times}$, but rather as a character of the real $n$-torus~$M_{S^1}$, and $\chi_i$ as a cocharacter on the subtorus $S_i \subset M_{S^1}$ defined by $\psi_i$.

The main theorem of this paper is a statement of homological mirror symmetry for the varieties $U$ and $U^\vee$.

\begin{Theorem}\label{thm:main}
 There is an equivalence of dg categories
 \[
 	\Coh(U)\cong \cW\big(U^\vee\big)
 \]
 between the category of coherent sheaves on $U$ and the wrapped Fukaya category of $U^\vee$.
\end{Theorem}

The slogan is, ``The blow-up-and-delete construction of Gross--Hacking--Keel is mirror to the Weinstein handle attachment of Example~\ref{ex:basic-surgery-highd}''.
We will deduce this theorem from the following more general result, which we can understand as a generalization of the calculation described in Example~\ref{ex:fund-calc}:

\begin{Theorem}\label{thm:main-general}
	Let $X$ be a smooth variety or DM stack, $D\xhookrightarrow{k} X$ a divisor, and $H\xhookrightarrow{i}D$ a~smooth subvariety of codimension $2$ in $X$. Let $\frX = \Bl_HX\setminus \widetilde{D}$ and let $\cC$ denote the
homotopy pushout of dg categories
\begin{equation}\label{eq:pushout}
\colim\bigl(\Coh(H)\stackrel{i^*}\longleftarrow\Coh(D)\stackrel{k_*}\longrightarrow\Coh(X)\bigr).
\end{equation}
Then there is an equivalence
\[
\cC\cong\Coh(\frX)
\] between $\cC$ and the category of coherent sheaves on $\frX$.
\end{Theorem}
\begin{proof}
We begin by writing down the functor $\cC\to \Coh(\frX)$. From the pushout presentation of $\cC$, such a functor may be specified by the data of functors $F_H\colon \Coh(H)\to \Coh(\frX)$ and $F_X\colon \Coh(X)\to\Coh(\frX)$, together with an equivalence between the pullbacks of $F_H$ and $F_X$ to $\Coh(D)$ along the maps in the pushout \eqref{eq:pushout}.

Consider the following diagram of all the spaces involved in our story, where $E$ is the exceptional divisor of the blowup, the vertical maps are the projections, and the horizontal maps are the inclusions:
\[
\xymatrix{
\frX\ar@{^{(}->}[r]^-j & \Bl_HX\ar[d]_p & &E\ar@{_{(}->}[ll]_-{\ovi}\ar[d]^q\\
& X & D\ar@{_{(}->}[l]^k & H.\ar@{_{(}->}[l]^-i
}
\]
The functors $F_H$ and $F_X$ will be defined as the compositions
\begin{gather*}
F_H = j^*\ovi_* q^*, \qquad F_X = j^*p^*.
\end{gather*}
The restriction to $D$ of $F_X$ is evidently given by $j^*p^*k_*$, and we must produce an isomorphism between this functor and the other restriction $j^*\ovi_*q^*i^*$. Note that before applying final composition with $j^*$, these two functors differ, since $p^*k_*$ takes sheaves on $D$ to their pullback to the blowup, whereas $\ovi_* q^*i^*$ takes sheaves on $D$ to their component lying along the exceptional divisor $E$. In other words, the difference between these two functors is a sheaf on the proper transform $\widetilde{D}$ of $D$. The final pullback $j^*$ will kill this component, forcing the two functors to agree.

Thus, we have a well-defined functor $F\colon\cC\to \Coh(\frX)$.
To understand this functor,
we will need to apply
Orlov's result \cite{Orl92} that $p^*\Coh(X)$ and $q^*\Coh(H)\otimes\cO_E(1)$ generate the category $\Coh(\Bl_H(X))$ of coherent sheaves on the blowup, and in fact they form a semi-orthogonal decomposion:
\begin{equation}\label{eq:semiorth}
\Coh(\Bl_H(X)) = \langle \Coh(H)\otimes \cO_E(-1), \Coh(X)\rangle.
\end{equation}
Note that after pulling back to $\frX$, we can ignore the tensor
product with $\cO_E(-1)$ in \eqref{eq:semiorth}, since deleting $\widetilde{D}$ makes $E$ affine. This makes clear why $F$ is essentially surjective: the image of the functor $F$ contains the pullbacks to $\frX$ of both components in the semi-orthogonal decomposition~\eqref{eq:semiorth} and hence contains all objects in $j^*\Coh(\Bl_H(X))$.

To see why the functor $F$ is in fact an equivalence of categories, recall that the
category $\Coh(\frX)$ is a localization of $\Coh(\Bl_H(X))$ obtained by killing the identity morphism of any object in $\Coh\big(\widetilde{D}\big)$, and observe that this is equivalent to the localization which identifies any object in the second term of~\eqref{eq:semiorth} of the form~$k_*\cF$, for some~$\cF$ in $\Coh(D)$, with the object $(iq)^*\cF\otimes \cO_E(-1)$ in the first term.

Now note that the
category obtained from the semi-orthogonal decomposition \eqref{eq:semiorth} by identifying the images of $\Coh(X)$ and $\Coh(H)$ is equivalent to the category obtained from the ortho\-go\-nal sum $\Coh(H)\oplus \Coh(X)$ by the same identification;
but $\Coh(H)\oplus \Coh(X)$ is the~co\-pro\-duct of the categories $\Coh(H)$ and $\Coh(X)$, and the identification just described is the coequalizer of the respective maps from $D$ to these categories. In other words, this presentation exhibits $\Coh(\frX)$ as the pushout \eqref{eq:pushout}, and $F$ as the canonical equivalence $\cC\cong \Coh(\frX)$.
\end{proof}

\begin{proof}[Proof of Theorem \ref{thm:main}]
	Recall from Corollary \ref{cor:b-pushout} that there is a commutative diagram
	\begin{equation}\label{eq:equivalence-of-diagrams}
\begin{split}&
\xymatrix{
	\cW\big(\oX_\Sigma^\vee\big)\ar[d]_{\rotatebox{90}{$\sim$}}&
	\cW\bigl(\bigsqcup_{i=1}^r T^*\big(S^1_{\chi_i}\times T_i\big)\bigr)\ar[r]\ar[l]\ar[d]_{\rotatebox{90}{$\sim$}}&
	\cW\bigl(\bigsqcup_{i=1}^r T^*D_{\chi_i}\times T_i\bigr)\ar[d]_{\rotatebox{90}{$\sim$}}\\
	\Coh\big(\oX\big)&\Coh\big(\oD\big)\ar[l]_{i_*}\ar[r]^{j^*}&\Coh(H)
}\end{split}
\end{equation}
whose vertical maps are equivalences,
where in the bottom row, $\oX\xhookleftarrow{i}\oD$ is the toric model for the log Calabi--Yau pair $(X,D)$, and $H\xhookrightarrow{j}\oD$ is the divisor defined by the characters specified by the seed data; and the spaces in the top row are the Liouville sectors identified in Proposition~\ref{prop:cover}.

Corollary \ref{cor:a-pushout} computes that the colimit of the top row of \eqref{eq:equivalence-of-diagrams} is the wrapped Fukaya cate\-gory~$\cW\big(U^\vee\big)$.
We now observe that the bottom row of \eqref{eq:equivalence-of-diagrams} satisfies the hypotheses of Theo\-rem~\ref{thm:main-general}, so that we may conclude from that theorem that the colimit of the bottom row is equivalent to the cate\-gory~$\Coh(\frX)$ of coherent sheaves on the space
\[
	\frX:=\Bl_H\oX\setminus\widetilde{\oD}
\]
obtained from the variety $\oX$ by blowing up along $H$ and then deleting the proper transform of the toric boundary divisor $\oD$. But by definition $\frX=U$ is the truncated cluster variety we are interested in.

The vertical equivalence between rows in the diagram \eqref{eq:equivalence-of-diagrams} induces an equivalence between the colimits of those rows; by the above discussion, this is the equivalence of categories
\[
	\Coh(U)\cong\cW\big(U^\vee\big)
\]
we wanted to establish.
\end{proof}

One consequence of the equivalence from Theorem \ref{thm:main} is that the category $\Coh\big(\oT_\CC\big)$ of coherent sheaves on the torus chart $\oT_\CC$ of the cluster variety $U$ corresponds under mirror symmetry to the category $\Loc\big(T^\vee\big)$ to the category of local systems on the dual torus $T^\vee$, which is a closed subset in the skeleton $\LL_{U^\vee}$ of~$U^\vee$.
A Weinstein neighborhood $T^*T^\vee$ of this torus is therefore a~Liouville subdomain of the mirror cluster variety $U^\vee$, and the Viterbo restriction of Proposition~\ref{prop:viterbo} gives a functor
\begin{equation}\label{eq:viterbo1}
\cW(U^\vee)\to \cW\big(T^*T^\vee\big).
\end{equation}
This functor is mirror to the pullback $\Coh(U)\to \Coh\big(\oT_\CC\big)$ of coherent sheaves on the cluster variety $U$ to a toric chart.

From this description, one can immediately guess what the symplectic mirror to the elementary transformations underlying cluster mutation should look like. This should be a procedure which exchanges the Weinstein presentation of $\LL$, as obtained from $T^*T^\vee$ by Weinstein disk attachments, to a different Weinstein presentation which begins with a different torus $T^*(T')^\vee$. We will see below that this guess is essentially correct.

\subsection{Mutations}\label{sec:mut}
We now discuss the effect of mutations in symplectic geometry. Local models for these mutations have been studied in \cite{Nad-mut,PT}, and in the 2-dimensional case a global description of these mutations is given in \cite{HK20,STW}.

As explained in \cite{STW}, from the perspective of a 2-dimensional Lagrangian skeleton $\LL$, cluster mutation corresponds to the Lagrangian disk surgery studied in \cite{Yau}.
\begin{Construction}
The local model for this surgery, depicted in Figure~$\ref{fig:disk-surg}$,
begins with a Lagrangian skeleton $\LL$ which can be presented as the union of a surface $\Sigma$ together with a disk $\DD$ attached to $\Sigma$ along a circle. There is a homotopy of Weinstein structures which collapses the disk $\DD$ and then expands it to produce a new Lagrangian skeleton $\LL'$ which is a union of $\DD$ with a surface $\Sigma'$.
Alternatively, if one folds up the bottom half of the cylinder to make a plane with the central disk, this transformation can be pictured via a cooriented circle in the plane shrinking to a point and then reexpanding with its opposite coorientation, as in Figure $\ref{fig:surg-gks}$.

\begin{figure}[!ht]
 \centering
 \includegraphics[width=6cm]{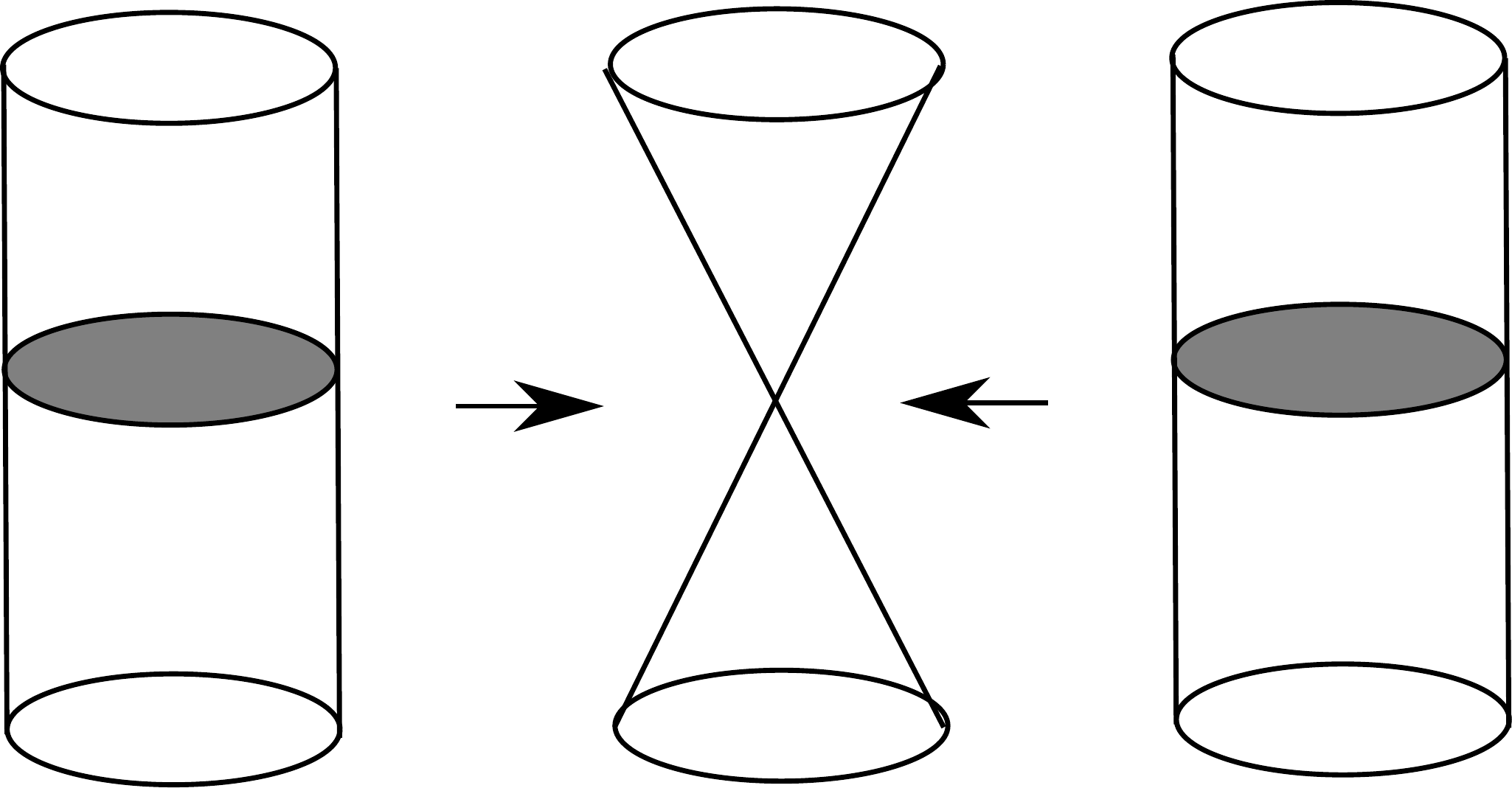}
 \caption{The local model for the Lagrangian disk surgery mirror to cluster mutation: a disk attached to a cylinder $\Sigma$ is collapsed and then expanded to produce a new cylinder $\Sigma'$.}
 \label{fig:disk-surg}
\end{figure}

\begin{figure}[!ht]
 \begin{center}
 \includegraphics[width=10cm]{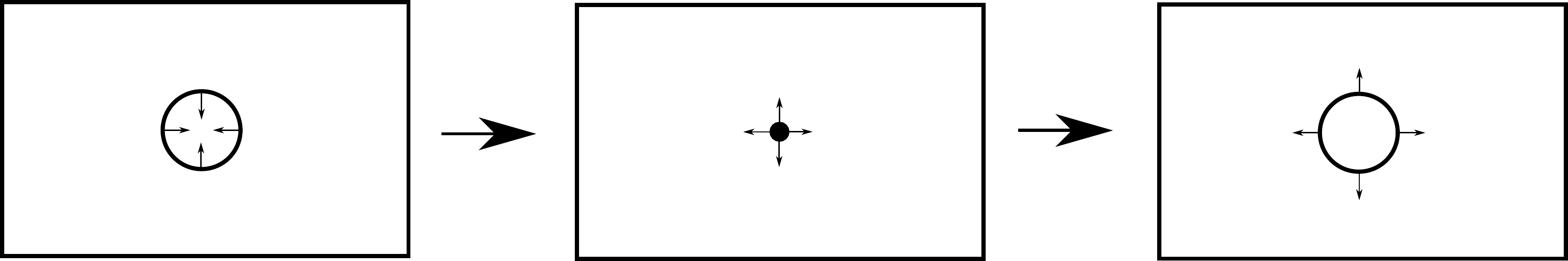}
 \end{center}
 \caption{The Lagrangian from Figure \ref{fig:disk-surg} is obtained by taking the cone on the Legendrian lift of the cooriented circle--or, in the central diagram, the whole conormal circle of the point. The coorientation of the circle changes as it shrinks to a point and regrows, although the Legendrian circles at infinity are all related by a contact isotopy.}
 \label{fig:surg-gks}
\end{figure}
\label{cons:basic-surgery}
\end{Construction}

Suppose that the surface $\Sigma$ also has other disks $\DD_1,\dots,\DD_r$ glued in, along Legendrians lifting cooriented circles $S_1,\dots,S_r\subset \Sigma$ which may intersect the boundary circle of $\DD$. If the above surgery implements mutation, then we should expect that it should modify the boundary circles~$S_i$ by a Dehn twist for each positively oriented intersection with the boundary circle of $\DD$. The results of \cite{STW} establish that this is the case. In our language, we can restate their result as follows:

\begin{Theorem}[{\cite[Theorem 5.7]{STW}}]\label{thm:stw-mutations}
Let $\LL$ be the Lagrangian skeleton of a Weinstein mani\-fold~$U^\vee$ obtained as in Construction $\ref{constr:2d-surgeries}$ by beginning with the cotangent bundle $T^*T$ of a~$2$-torus~$T\cong T^2$ and attaching Weinstein handles along Legendrian lifts of cooriented circles $S_1,\dots,S_k\subset T$.

Applying the disk surgery described in Construction $\ref{cons:basic-surgery}$ to the $i$th disk $\DD_i$ results in a new Lagrangian skeleton $\LL'$ of the same Weinstein manifold, presented by beginning with the cotangent bundle $T^*T'$ of a $2$-torus $T'\cong T^2$ by handle attachments along Legendrian lifts of cooriented circles $S_1,\dots,S_k\subset T'$, where $S_i'=S_i$ with opposite coorientation, and for $j\neq i$ we have
\[
	S_j' = \begin{cases} \tau_iS_j&\text{if}\ \langle S_j, S_i\rangle >0,\\
	 S_j&\text{otherwise},\end{cases}
\]
where we write $\langle -,-\rangle$ for the oriented intersection number and $\tau_i$ for the Dehn twist about $S_i$.
\end{Theorem}

The paper \cite{STW} also checks that the respective tori $T$, $T'$ contained in the skeleton before and after a disk surgery are related by cluster mutation. Continuing with the notation of Theorem~\ref{thm:stw-mutations}, Viterbo restriction \eqref{eq:viterbo1} gives a pair of functors
\begin{equation}
 \Coh\big(T_\CC^\vee\big)\cong\cW(T^*T)\gets \cW\big(U^\vee\big)\to\cW(T^*T')\cong \Coh\big( (T')^\vee_\CC\big),
 \label{eq:viterbo2}
\end{equation}
which we would like to understand as mirror to a birational map $T_\CC^\vee\dashrightarrow(T')^\vee_\CC$ between the complexified dual tori. This map is studied in \cite{STW} by determining what it does to skyscraper sheaves in $\Coh\big(T_\CC^\vee\big)$, which correspond to finite-rank local systems on~$T$. If $\cE$ is a local system on $T$, we can imagine $\cE$ as a vector space $\cE$ together with an endomorphism $\cE_\gamma$ for each path $\gamma$ in~$T$.

Now note that on finite-rank local systems (i.e., objects of the {\em compact} Fukaya category of~$T^*T$), the right adjoints in~\eqref{eq:viterbo2} are defined, so we can obtain a diagram
\begin{equation}
 \Tors\big(T^\vee_\CC\big)\cong\Loc^{\rm fin}(T)\to \cW\big(U^\vee\big)\gets\Loc^{\rm fin}(T')\cong \Tors\big( (T')^\vee_\CC\big),
\label{eq:viterbo-adj}
\end{equation}
and we can ask when a local system on $T$ maps into the image of the right-hand functor -- i.e., to an object of $\cW\big(U^\vee\big)$ represented by a local system on $T'$.

\begin{Theorem}[{\cite[Theorem 4.16]{STW}}]\label{thm:stw-birational}
 Let $\cE$ be a finite-rank local system on $T$. Then $\cE$ is equivalent in $\cW\big(U^\vee\big)$ to a local system $\cE'$ on $T'$ if and only if the monodromy of $\cE$ around $S_i$ does not have 1 as an eigenvalue.

 Suppose this is the case. If $\gamma\subset T^2$ is a path not intersecting $S_i'$, then $\cE'_\gamma=\cE_\gamma$, and if~$\gamma$ is a path which crosses $S_i'$ against its coorientation, in a neighborhood of the crossing point the holonomy is modified to $\cE'_\gamma = \on{Id}-\cE_{S_i}$, where $S_i$ is treated as an oriented path, where its coorientation points rightward. $($The identification of paths in $T$ with paths in $T'$ is given by a~Dehn twist about $S_i.)$
\end{Theorem}

These results generalize easily to a higher-dimensional Weinstein manifold $U^\vee$ obtained as in Construction \ref{cons:mirror} by Weinstein handle attachments along Legendrian lifts of codimension-1 subtori $S_1,\dots,S_r$ in an $n$-torus $T$.

\begin{Construction}\label{cons:basic-surgery-highd}
 Let $T$, $S_i$ as above. Near the subtorus $S_i$, the torus $T$ splits as $T^2\times T^{n-2}$, where the first $T^2$ is spanned by $S{\chi_i}$ $($where $\chi_i$ is the cocharacter of $S_i$ determining the handle attachment as in Construction~$\ref{cons:mirror})$ and a circle not contained in $S_i$.
We can thus take the product with $T^{n-2}$ of the disk surgery of Construction $\ref{cons:basic-surgery}$ above to obtain an analogous operation on $T=T^n$. The result is a new Lagrangian which is the union of an n-torus $T'$ with Weinstein handles attached to the Legendrian lifts of codimension-$1$ subtori $S_i'$.
\end{Construction}

 Theorems \ref{thm:stw-mutations} and~\ref{thm:stw-birational} can be applied directly to this setting to describe the new torus $T'$ and its subtori $S_i'$ just as in the 2-dimensional case. In the language of cluster seeds, we can summarize this discussion as follows:
\begin{Theorem}
 Let $T$, $S_1,\dots,S_r$, as in Construction $\ref{cons:mirror}$, corresponding to a seed $\frs$. Then~$T'$, $S_1',\dots,S_r'$ as in Construction $\ref{cons:basic-surgery-highd}$ correspond to the $i$th mutation $\mu_i\frs$ of $\frs$. Moreover, the birational map defined by~\eqref{eq:viterbo-adj} agrees with the cluster transformation for this mutation.
\end{Theorem}

\subsection*{Acknowledgements}The authors are grateful to Roger Casals and Dmitry Tonkonog, for explanations about nodal trades; Andrew Hanlon, for discussions about the GHK construction; Vivek Shende and Harold Williams, for discussions about the papers~\cite{STW,STWZ}; Harold again, for several helpful comments on a draft of this paper; and
Denis Auroux and Daniel Pomerleano, for comments on a now-deleted section of this paper.
%a correction to Section~\ref{sec:trades} and for the content of Remark~\ref{rem:crit-filt}.
BG is supported by an NSF postdoctoral fellowship, DMS-2001897.
\textit{Added in revised version: the authors would also like to thank Peng Zhou for pointing out an error in the discussion of K\"ahler potentials in a previous version of this paper.}

\pdfbookmark[1]{References}{ref}
\LastPageEnding

\end{document}